\documentclass{amsart}

\usepackage{cite}
\usepackage{enumitem}
\usepackage{graphicx}
\usepackage{amssymb,latexsym}
\usepackage{amsmath}
\usepackage{amsthm}
\usepackage{color, soul}
\usepackage{hyperref}
\usepackage[dvipsnames]{xcolor}
\usepackage{tikz-cd}
\usepackage{multicol}
\usepackage{enumitem}
\usepackage{todonotes}
\usepackage{wasysym}
\usepackage{pgf}
\usepackage{tikz}
\usetikzlibrary{arrows,automata}
\usepackage{setspace} 
\usepackage{subcaption}
\usepackage{caption}

\newtheorem{theorem}{Theorem}[section]
\newtheorem{lemma}[theorem]{Lemma}
\newtheorem{prop}[theorem]{Proposition}
\newtheorem{example}[theorem]{Example}

\newtheorem{remark}[theorem]{Remark}
\newtheorem{corollary}[theorem]{Corollary}
\newtheorem{definition}[theorem]{Definition}
\numberwithin{figure}{section}
\numberwithin{equation}{section}

\newcommand{\diamp}{\mathrm{diam}_{p}}
\newcommand{\diamr}{\mathrm{diam}_{0}}

\newcommand{\dimp}{\mathrm{dim}_{\mathrm{H},p}}
\newcommand{\dimr}{\mathrm{dim}_{\mathrm{H},0}}
\newcommand{\dimlc}{\mathrm{dim}_{\mathrm{loc}}}

\newcommand{\StandardClause}{Let $\F = \{F_i\}_{i=1}^n$ be a $p$-adic iterated function system, where each $F_i(x) = (-1)^{b_i} p^{k_i} x + d_i$, with $b_i \in \{0,1\}$, $k_i \in \Z$, $k_i \geq 1$, and $d_i \in \Z_p \cap \Q$.
Let $K$ be the $p$-adic self-similar set associated to $\F$.
} 

\def\Z{{\mathbb Z}}
\def\R{{\mathbb R}}
\def\Q{{\mathbb Q}}
\def\N{{\mathbb N}}
\def\D{{\mathcal D}}
\def\A{{\mathcal A}}
\def\Ad{{\overline{\mathcal A}}}

\def\F{{\mathcal F}}

\def\Ld{{\overline{\mathcal L}}}
\def\Ln{{\mathcal L}}

\begin{document}

\title{Self-similar sets and self-similar measures in the $p$-adics.}
\author{Kevin G. Hare}
\address{Department of Pure Mathematics \\
          University of Waterloo \\
          Waterloo, Canada}
\email{kghare@uwaterloo.ca}
\thanks{Research of K.G. Hare was supported, in part, by NSERC Grant 2019-03930}
\author{Tom\'a\u{s} V\'avra}
\address{Charles University, Faculty of Mathematics and Physics, Department of Algebra, Sokolovská 83, 18600 Praha 8, Czech Republic}
\address{Department of Pure Mathematics \\
          University of Waterloo \\
          Waterloo, Canada}
\email{tvavra@uwaterloo.ca, vavrato@gmail.com}
\thanks{Research of T. V\'avra was supported, in part, by the University of Waterloo and NSERC Grant 2019-03930}

\maketitle

\begin{abstract}
In this paper we investigate $p$-adic self-similar sets and $p$-adic self-similar measures.
We show that $p$-adic self-similar sets are $p$-adic path set fractals, and that the converse is not necessarily true.
For $p$-adic self-similar sets and $p$-adic self-similar measures, we show the existence of a unique essential class.
We show that, under mild assumptions, the decimation of $p$-adic 
    self-similar sets is maximal.
For $p$-adic self-similar measures, we show that many results involving
    local dimension are similar to those of their real counterparts,
    with fewer complications.
Most of these results use the additional structure of self-similarity,
    and are not true in general for $p$-adic path set fractals.
\end{abstract}

%%%%%%%%%%%%%%%%%%%%%%%%%%%%%%%%%%%%%%%%%%%%%%%%%%%
\section{Introduction}

Let $\Q_p$ be the $p$-adic numbers.
We say that $F(x) = a x + d$ with $F: \Q_p \to \Q_p$ is a linear contraction if $|F(x) - F(y)|_p < |x-y|_p$ for all $x \neq y$.
Let $\F = \{F_1, F_2, \dots, F_n\}$ be a finite set of linear 
    contractions.
We say that $\F$ is a {\em $p$-adic iterated function system}.
There exists a unique non-empty compact set $K \subseteq \Q_p$, called the {\em attractor} or {\em $p$-adic self-similar set} such
    that \[ K = \bigcup_{i} F_i(K).\]

In this paper we will focus on the case where the 
    linear contractions are of the form $F(x) = \pm p^k x + d$ for 
    some $k \geq 1$ and $d \in \Z_p \cap \Q$.

These object have been studied in \cite{JKLW, KRC, LH11, LHvF, LH09, LiQiu}.
Abram, Lagarias and Slonim, in a series of papers, explored $p$-adic path set fractals, \cite{AbramLagarias14a, AbramLagarias14b, AbramLagariasSlonim21}.  
Consider an automaton given by a directed graph $G$ with vertices $v_1, \dots, v_n$.
To each edge of this graph we associate an output from $\{0, 1, \dots, p-1\}$.
Then the set of $p$-adic numbers associated to this directed graph from a starting vertex $v_1$ have $p$-adic representations given by the set of infinite words accepted by this  automaton.  
This set of these $p$-adic numbers is called a {\em $p$-adic path set fractal}.

In this paper we will study the relationship between $p$-adic path set fractals and $p$-adic self-similar sets.
Notation used through this paper is introduced in Section \ref{sec:notation}.
In Section \ref{sec:self-similar are path set} we will show that, under reasonable conditions, that $p$-adic self-similar sets are $p$-adic path set fractals.
We will show that the set of $p$-adic path set fractals is strictly larger than the set of $p$-adic self-similar sets.
In Section \ref{sec:full} we consider the structure of a $p$-adic path set fractal with positive Haar measure.
We show in Section \ref{sec:essential class} that, under reasonable conditions, that a $p$-adic self-similar set always has a unique essential class. 
This is not true in general for $p$-adic path set fractals.
In Section \ref{sec:decimation} we discuss decimation for $p$-adic self-similar sets.
The study of $p$-adic self-similar measures and local dimensions is given in Section \ref{sec:measures}, and is analogous to self-similar measures and local dimensions on $\R$.
Finally, in Section \ref{sec:conclusions} we make some final comments and raise some open questions.

%%%%%%%%%%%%%%%%%%%%%%%%%%%%%%%%%%%%%%%%%%%%%%%%%%%
\section{Notation}
\label{sec:notation}
\subsection{\texorpdfstring{$p$}{p}-adic numbers}
Fix a prime number $p$. The $p$-adic valuation $\nu:\Q^*\to\Q$ is defined as $\nu\left(\tfrac ab\right) = r$ where $\tfrac ab = p^r \tfrac{a'}{b'}$ with $p$ being coprime to both $a',b'.$ The $p$-adic absolute value $|\cdot|_p:\Q\to\Q$ is defined as $|x|_p = p^{-\nu(x)}$ with $|0|_p = 0.$ The name absolute value refers to certain axioms being satisfied, most notably $|xy|_p = |x|_p|y|_p$ and the (strong) triangle inequality $|x+y|_p \leq \max\{|x|_p,|y|_p\}.$  Moreover, we have that $|x+y|_p = \max\{|x|_p,|y|_p\}$ whenever $|x|_p\neq |y|_p.$ The $p$-adic absolute value is an ultra-metric inducing a topology on $\Q.$ The topological closure of $\Q$ with respect to $|\cdot|_p$ is called the field of $p$-adic numbers, denoted by $\Q_p.$ 

The standard way of expressing $p$-adic numbers is through their $p$-adic expansion. It is an expression of the form $\sum_{i=k}^\infty a_ip^i$ where $k\in\Z$ and $a_i\in\{0,\dots,p-1\}$ for all $i\geq k.$ A $p$-adic expansion is eventually periodic if and only if the expanded number is rational, in which case the value can be computed through the closed formula for the sum of geometric series. The set of $p$-adic integers is 
$$\Z_p = \left\{\sum_{i\geq 0}a_ip^i : a_i\in\{0,\dots,p-1\}\right\}.$$
The integers $\Z$ lie in $\Z_p$, with non-negative integers having finite representations, and negative integers having eventually periodic expansions.

See \cite{MR0169847, MR1488696, MR754003} for a more complete introduction to $p$-adic numbers.

%%%%%%%%%%%%%%%%%%%%%%%%%%%%%%%%%%%%%%%%%%%%%%%%%%%
\subsection{Automata theory}
A finite automaton $\A$ on the finite set of symbols $A$ is given by a finite set of states $Q$, by transitions $E\subseteq Q\times A\times Q,$ and by an initial state $\mathbf{i}\in Q$.  In this paper, restrict our attention to finite automaton where all states are accepting states.  We denote the set of all finite words over an alphabet $A$ as $A^*$.  A finite word $s=a_1a_2\dots a_n\in A^*$ is said to be accepted by $\A$ if there exist transitions $(\mathbf{i},a_1,q_1),(q_1,a_2,q_2),\dots,(q_{n-1},a_{n},q_n)$ all belonging to $E.$ A subset of $L\subseteq A^*$ (usually called a language) is said to be recognized by a finite automaton $\A$ if $\A$ accepts precisely the elements of $L.$ A \emph{deterministic finite automaton} (DFA) has the property that for every $q\in Q$ and $a\in A,$ there is at most one $q'\in Q$, such that $(q,a,q')\in E.$ A {\em non-deterministic finite automaton} (NDFA) is a finite automaton that is not deterministic.  It is a classical result that if a language is recognized by a non-deterministic finite automata, then there exists a deterministic finite automata that also recognizes this language. We can associate to a finite automata a finite graph on vertices $Q$. For each $(q, a, q') \in E$ we associate a directed edges $(q, q')$ labeled by $a$. The automaton is then deterministic if for each vertex $q$ and each $a\in A,$ at most one edge labeled $a$ leaves $q.$

An infinite word $a_1a_2\dots$ over $A$ is said to be accepted by $\A$ if all prefixes $a_1\dots a_n$ are accepted by $\A$. We can see that if a language of infinite words is recognized by a non-deterministic automaton, then it is recognized by a deterministic one. Note that the usual notation of automata for infinite words requires an acceptance condition. They might be for instance that an accepting state is visited infinitely many times (B\"uchi automaton), or that every infinitely often visited state belongs to accepting set (Muller automaton). It is interesting that in the former case (but not in the latter) the distinction between deterministic and not-deterministic matters. Thus, the case we consider may be seen as either a B\"uchi or a Muller automaton with all states being accepting.

A deterministic finite automaton can be constructed from a non-deterministic finite automaton in the following way. Let $Q$ be the set of states of the NDFA.  The set of states of the DFA, which we will denote $\overline Q$, are subsets of $Q$, and the initial state is $\{\mathbf{i}\}$. For a letter $a\in A,$ the transition from a state $\overline q_1 \in \overline Q$ labeled by $a$ is the state $\overline q_2 = \bigcup_{q\in \overline q_1} \{q'\in Q : (q,a,q')\in E)\}$. In other words, for any path $a_1a_2\dots a_n$, the state reached by this path from $\{\mathbf{i}\}$ is all the states of $Q$ that can be reached by this path in the non-deterministic version. It is not hard to see that the new automaton recognizes the same language.

A finite state transducer is defined similarly to a finite state automaton.  
The key difference is that, in additional to an alphabet $A$ there is an output alphabet
    $B$.
The transitions $E \subseteq Q \times A \times Q \times B$.
For a four tuple $(q_1, a, q_2, b)$ we interpret $q_1$ as the start state of a transition,
    $q_2$ as an end state of a transition.
We interpret $a$ as the input alphabet and $b$ as the output alphabet.
This allows us to ``read in'' an infinite word accepted by $\A$ and output
    an infinite word in $B^{\N}$.

For a more complete introduction to automata theory with infinite strings, see \cite{PerrinPin}.

%%%%%%%%%%%%%%%%%%%%%%%%%%%%%%%%%%%%%%%%%%%%%%%%%%%
\subsection{Non-negative matrices}
For a DFA $\A$ with $n$ states, we define its adjacency matrix $T=T_{i,j}$ for $0\leq i,j\leq n$, where $T_{i,j}$ is the number of transitions from the state $q_i$ to $q_j.$ Defined this way, the $i,j$-th component of $T^k$ is the number of different words $s$ of length $k$ from $q_i$ to $q_j$ labeled by $s$. We immediately get that $T$ is non-negative. If the graph of $\A$ is strongly connected, we see that $T$ is irreducible. If $\A$ is not irreducible we can permute $T$ into a block triangular shape, where the diagonal blocks correspond to strongly connected components of $\A$.

The dominate eigenvalue value of a matrix $M$ is known as the {\em spectral radius} and is denoted $\rho(M)$.
For $M_1$ and $M_2$ both $n \times n$ matrices, we have that if $0\leq M_1 < M_2,$ then there is an inequality of the spectral radii $\rho(M_1)\leq \rho(M_2),$ and in the case that $M_2$ is irreducible, we have $\rho(M_1)<\rho(M_2).$ We will use this fact later.

%%%%%%%%%%%%%%%%%%%%%%%%%%%%%%%%%%%%%%%%%%%%%%%%%%%
\subsection{Hausdorff dimension}

For any subset $X \subseteq \Q_p$ we define the diameter of $X$ as 
    \[ \diamp(X) := \sup_{x,y \in X} |x-y|_p .\]
For any $\delta > 0$ and any $d \geq 0$ we define 
\begin{equation*}
H_{\delta,p}^d(X) := \inf \left\{ \sum_{i=1}^\infty \diamp(X_i)^d: X \subseteq \bigcup X_i, \diamp(X_i) < \delta\right\}. 
\end{equation*}
We define the {\em outer measure} $H_{p}^d(X)$ as \[ H_{p}^d(X) = \lim_{\delta \to 0} H_{\delta,p}^d(X) .\]

The function $H_{p}^d(X)$ is decreasing in $d$.
For most $d$ the value of $H_{p}^d(X)$ is either $0$ or infinity.  In fact there is at most
    one value of $d$ where it can have a non-zero finite value. 
    
For $X \subset \Q_p$, we define the {\em Hausdorff dimension}, 
    $\dimp(X) := \inf \{d \geq 0: H_{p}^d(X) = 0\} = \sup \{d \geq 0: H_{p}^d(X) = \infty \}$.
For any $Y\subseteq K$ we define the {\em Hausdorff measure} $\mu_{K, p}$ with respect to $K$ as $\mu_{K,p}(Y)=H_p^{\dimp(K)}(Y)$.
We will show that when $K$ is a $p$-adic self-similar set that $0 < \mu_{K,p}(K) < \infty$. For $K \subset \R$, we define $\diamr, H_{\delta,0}^d, H_{0}^d, \dimr$ and $\mu_{K,0}$ in an analogous way.
See \cite{AbramLagarias14a, KRC} for further details.

%%%%%%%%%%%%%%%%%%%%%%%%%%%%%%%%%%%%%%%%%%%%%%%%%%%
\section{Path set fractals and \texorpdfstring{$p$}{p}-adic self-similar sets}
\label{sec:self-similar are path set}

\StandardClause
For each $x \in K$ there exists an infinite sequence $a_0, a_1, \dots$ such that $x = \lim_{k \to \infty} F_{a_1} \circ F_{a_2} \circ \dots \circ F_{a_k}(0)$.
We say that $\sigma$ is an address of $x$. 
It is worth noting that $x$ may have more than one address.
It is clear that 
\begin{equation}\label{expression} x = \sum_{j=1}^\infty (-1)^{\sum_{i=1}^{j-1} b_{a_i}} p^{\sum_{i=1}^{j-1} k_{a_i}} d_{a_j}. \end{equation}
As such, we have:
\begin{lemma}\label{lem:basic}
The points of $K$ are precisely 
\begin{equation} \left\{ \sum_{j=1}^\infty (-1)^{\sum_{i=1}^{j-1} b_{a_i}} p^{\sum_{i=1}^{j-1} k_{a_i}} d_{a_j}: (a_j) \in \{1,2,\dots, n\}^{\N} \right\}.\label{eq:K} \end{equation}
\end{lemma}
The proof consists only of the verification that $K=\bigcup F_i(K).$
It is worth noting that the above sums are $p$-adic expansions
    only if $k_i = 1$, $d_i \in \{0,1, \dots, p-1\}$ and $b_i = 0$ for all $i$.
Otherwise, the above sums are well defined, but will need some rewriting
 to give a $p$-adic number in the standard form.

Consider an automaton given by a directed graph $G$ with vertices $v_1, \dots, v_n$.
To each edge of this graph we associate an output from $\{0, 1, \dots, p-1\}$.
Then the set of $p$-adic numbers associated to this directed graph from a starting vertex $v_1$ are the set of infinite words accepted by this  automaton.  
These fractals are known as {\em $p$-adic path set fractals}, and were 
    first explored in 
 \cite{AbramLagarias14a, AbramLagarias14b, AbramLagariasSlonim21}.  
 It is shown that the Hausdorff dimension of $p$-adic path set fractals exists, and is equal to $\frac{\log(\lambda)}{\log(p)}$, where $\lambda$ is the dominant eigenvalue of the adjacency matrix of the $p$-adic path set fractal.

We first give a simple example when the $b_i =0$ and $k_i = 1$, to show how to rewrite points in $K$ from the form \eqref{eq:K} into a standard 
   $p$-adic representation.
This example will also be shown to be a $p$-adic path set fractal.
    
\begin{example} \label{ex:3adic}
Consider $p=3$ and the $p$-adic IFS $\{A, B, C\}$ with
    $A(x) = 3 x + 0, B(x) = 3 x + 1$ and $C(x) = 3 x + 3$.
We have that $K = \left\{\sum 3^i d_i: d_i \in \{0,1,3\} \right\}$. We can start rewriting $\sum_{i\geq 0} 3^id_i$ from the least significant digits $d_0,d_1,d_2,\dots$. We note that the digit $3$ is not an allowable digit in a $3$-adic representation of a number. Because $3p^k + d_{k+1}p^{k+1} = (1+d_{k+1})p^{k+1},$ instead of the digit $3$ we output $0$ at the position of $p^k$, and remember the carry of $1$. If $d_{k+1}\in\{0,1\},$ we can resolve the carry. In case $d_{k+1}=3,$ the carry propagates further, as we use $4p^{k+1} = p^{k+1} + p^{k+2}$ where $p^{k+2}$ is the new carry.
This procedure is visualized in Figure \ref{fig:013a}.

\begin{figure}
\centering
\begin{subfigure}[t]{0.4\textwidth}
\begin{tikzpicture}[->,>=stealth',shorten >=1pt,auto,node distance=2.8cm,
                    semithick]

  \node[initial,state] (A)                    {$0$};
  \node[state]         (B) [below of=A] {$1$};

  \path (A) edge [bend right=50]  node[left] {C/0} (B)
            edge [loop above] node {A/0} (A)
            edge [loop right] node {B/1} (A)

        (B) edge [bend right=20] node {A/1} (A)
            edge [bend right=50] node[right] {B/2} (A)
            edge [loop right] node {C/1} (B);
\end{tikzpicture}
\caption{Transducer} \label{fig:013a}
\end{subfigure}
\begin{subfigure}[t]{0.35 \textwidth}
\begin{tikzpicture}[->,>=stealth',shorten >=1pt,auto,node distance=1.8cm,
                    semithick]

  \node[initial,state] (A)                    {$\{0\}$};
  \node[state]         (B) [below of=A] {$\{0,1\}$};

  \path (A) edge [bend right=50]  node[left] {0} (B)
            edge [loop above] node {1} (A)

        (B) edge [bend right=50] node[right] {2} (A)
            edge [loop right] node {0} (B)
            edge [loop below] node {1} (B);

\end{tikzpicture}
\caption{DFA} \label{fig:013b}
\end{subfigure}
\caption{Transducer and DFA for Example \ref{ex:3adic}} \label{fig:013}
\end{figure}
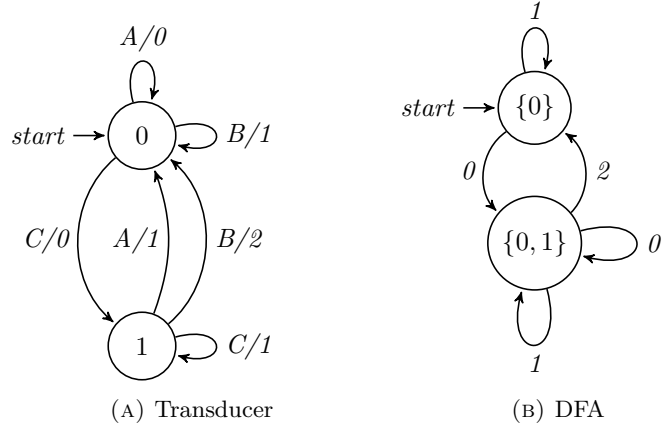

Each edge of the transducer associated to the $p$-adic self-similar set has both
    and input and and output symbol.
To describe the attractor as a $p$-adic path set fractal, we associate to each edge the output symbol only.
From this, we will obtain a non-deterministic finite state machine recognizing the $p$-adic expansions of elements of $K$.  This is visualized in Figure \ref{fig:013b}.
\end{example}

The idea from the previous example is the essence behind the following theorem.

\begin{theorem}\label{thm:main}
\StandardClause
Then the rewriting from elements given in equation \eqref{eq:K} to standard $p$-adic expansions $\sum a_i p^i$ with $a_i \in \{0,\dots,p-1\}$ is realized by a finite state transducer.

In particular, the elements of $K$ are recognized by a finite state automaton, thus the Hausdorff dimension of $K$ is computable.
\end{theorem}

\begin{proof}
We will consider three sub-cases, which have subtle differences in how
    they interact with the transducer.
\begin{enumerate}
\item $F_i(x) = p x + d_i$ for $d_i \in \Z_p \cap \Q,$
\label{case:d}
\item $F_i(x) = p^{k_i} x + d_i$ for $k_i \geq 2,$
\label{case:p^2}
\item $F_i(x) = - p^{k_i} x + d_i.$
\label{case:negative}
\end{enumerate}

{\bf Case \eqref{case:d}:}
Assume all maps are of the form $F_i(x) = p x + d_i$ for $d_i \in \Z_p \cap \Q$.
For $x\in\Z_p$, let $D(x)\in\{0,\dots,p-1\}$ be the least significant digit in the $p$-adic expansion of $x.$ 
That is, if $x = \sum_{i=0}^\infty a_i p^i$ then $D(x) = a_0$.
Notice that in this case, the $n$-th digit of the $p$-adic expansion of $\sum_{i\geq 0} p^id_i$ is determined by $d_1,\dots,d_n$ only.

We will now construct a non-deterministic transducer. The initial state is $\mathbf{i}=0\in\Q \cap \Z_p$ and for any input symbol $d\in\D,$ the output symbol is $D(d)$ and the new state is $s' = \tfrac1p(d-D(d)).$ Namely, the new state is the carry, i.e. the value that carries to higher powers of $p.$ In general, the states are a subset of $\Q\cap \Z_p$. From each state $s$, there are $\#\D$ edges. The transition from $s$ while reading $d$ is outputting symbol $D(s+d)$ and changing the state to $s' = \tfrac1p(s+d-D(s+d)).$ With $d$ and $D(x)$ being from a finite set, there are only finitely many states.

{\bf Case \eqref{case:p^2}:}
The above technique needs to be slightly modified when some maps are of the form $F_i(x) = p^{k_i} x + d_i$ for $k_i \geq 2$.
Define $D_k(x) = \{0, 1, \dots, p^{k}-1\}$ to be the least 
    $k$ significant digits of the $p$-adic expansion of $x$.
That is, $D_k(a_0 + a_1 p + \dots) = a_0 + a_1 p + \dots + a_{k-1} p^{k-1}$.
From equation \eqref{expression} we see that the $k_i$-block of $\sum_{i \geq 0} p^{\sum_{j=0}^{i} k_j} d_i$ starting at $\sum^{i-1} k_j + 1$ position is determined by $d_1, d_2, \dots, d_{i}$ only. 

The key difference is that the output symbol is a block of length $k$ given by $D_k(s+d)$ and the new state is defined as $s' = \frac{1}{p^k} (d+s - D_k(d+s))$.
That is, if $D_k(s+d) = a_0 + a_1 p + \dots + a_{k-1} p^{k-1}$ then we output 
    $a_0, a_1, \dots, a_{k-1}$.

{\bf Case \eqref{case:negative}:}
The above technique needs again a modification if some of the maps
     have negative contractions.
In this case, we will keep track in the state if we have had an even or odd number of maps with negative contractions within the first $n$ terms. 
It is sufficient to keep track of the parity of $\sum_{i=1}^{j-1} b_{a_i}$ as can be seen from equation  \eqref{expression}.
If we have had an even number, then things proceed as before.
If we have had an odd number, we must modify our output to be $D_k(-(d+s))$, as well as modifying are new state at
$s' = \frac{1}{p^k} (D_k(d+s) - d - s)$.

As the transducer acts on $\D$, we obtain a non-deterministic automaton by forgetting the input symbols. Moreover, the language of a non-deterministic finite automaton is recognized by a deterministic one.
We can use standard techniques for converting a non-deterministic automaton to a minimal deterministic automaton.  See for example \cite{AlloucheShallit}.
\end{proof}

\begin{corollary}
A $p$-adic self-similar set with contractions of the form $x \to \pm p^k x + b$ for $b \in \Z_p \cap \Q$ is a $p$-adic path set fractal.
\end{corollary}

\begin{example}
\label{ex:pm}
Consider a $5$-adic self-similar set given by the two maps
\begin{align*}
    & A: x \mapsto -5 x, \\
    & B: x \mapsto 5^2 x + 1/2.
\end{align*}

To construct our non-deterministic transducer, we start in state $\mathbf{i} = (0, +)$ and consider the actions of maps $A$ and $B$ in the positive direction with a carry of $0$.

In the case of the map $A$, we see that the output will be $0$ and the new state will be $(0, -)$.  The ``$-$'' is a change of signs that is a result of the map being a negative contraction.

We indicate this on Figure \ref{fig:pma} as a directed edge from $(0,+)$ to $(0,-)$ and labelling the edge $A/0$.  That is, input $A$ result in output $0$ and a change of state from $(0,+)$ to $(0,-)$.

Next we consider the action of $B$ on state $(0,+)$.  
It is worth noting that $1/2 = 3 + 2 \cdot 5 + (-1/2) \cdot 5^2$.
This means that the output will be $3,2$ (as we have a block of length $2$), and the new state will be $(-1/2, +)$.

In a similar fashion, we consider the action of $A$ and $B$ on the two new states $(0,-)$ and $(-1/2, +)$.  This results in four more directed edges, and one new state to consider, namely $(-1/2, -)$. 
We repeat this process on any new states that are found, until such time as no new states are found.

These are summarized in Figure \ref{fig:pma}.  

In general, the resulting transducer will take an input language from $\D^{\N}$ and produce a $p$-adic number.
In this case $\D = \{A, B\}$ and we produce a $5$-adic number.  
To produce the $p$-adic path set fractal (via an automaton), we replace input/output combinations on an edge with the output only.  
In addition, for any output which is a block of length $k$, we insert $k-1$ vertices to expand this out to a automaton with output of length one for each edge.
In general this resulting automaton need not be deterministic, (although in this case it is).  
If it is a non-deterministic automaton, we can convert this to the minimal deterministic automaton using a standard process. 
See Figure \ref{fig:pmb}.

\begin{figure}
\centering
\begin{subfigure}{0.5\textwidth}
\begin{tikzpicture}[->,>=stealth',shorten >=1pt,auto,node distance=3cm,
                    semithick]

  \node[initial,state] (A)                    {$(0,+)$ };
  \node[state]         (B) [right of=A] {$(0,-)$};
  \node[state]         (C) [below of=A] {$(-1/2,+)$};
  \node[state]         (D) [right of=C] {$(-1/2,-)$};

  \path (A) edge [bend left]  node {A/0} (B)
            edge [bend left]  node[left] {B/3,2} (C)
        (B) edge [bend left]  node {A/0} (A)
            edge [bend left]  node[right] {B/2,2} (D)
        (C) edge [bend left]  node[left] {B/0,0} (A)
            edge [bend left]  node {A/3} (D)
        (D) edge [bend left]  node {A/2} (C)
            edge [bend left]  node[right] {B/1,0} (B);
\end{tikzpicture}
\caption{Transducer}
\label{fig:pma}
\end{subfigure}
\begin{subfigure}{0.5\textwidth}
\begin{tikzpicture}[->,>=stealth',shorten >=1pt,auto,node distance=2cm,
                    semithick]

  \node[initial,state] (A) {$(0,+)$ };
  \node[state]         (B) [right of=A] {$(0,-)$};
  \node[state]         (AC) [below of=A] {b};
  \node[state]         (CA) [left of=AC] {a};
  \node[state]         (C) [below of=AC] {$(-1/2,+)$};
  \node[state]         (DB) [below of=B] {c};
  \node[state]         (BD) [right of=DB] {d};
  \node[state]         (D) [right of=C] {$(-1/2,-)$};

  \path (A) edge [bend left]  node {0} (B)
            edge   node[left] {3} (AC)
        (AC) edge  node[left] {2} (C)
        (B) edge [bend left]  node {0} (A)
            edge [bend left]  node[right] {2} (BD)
        (BD) edge [bend left] node[right] {2} (D)
        (C) edge [bend left]  node[left] {0} (CA)
            edge [bend left]  node {3} (D)
        (CA) edge [bend left] node[left] {0} (A)
        (D) edge [bend left]  node {2} (C)
            edge  node[right] {1} (DB)
        (DB) edge  node[right] {0} (B);
\end{tikzpicture}
\caption{DFA}
\label{fig:pmb}
\end{subfigure}
\caption{Transducer and DFA for Example \ref{ex:pm}}
\label{fig:pm}
\end{figure}
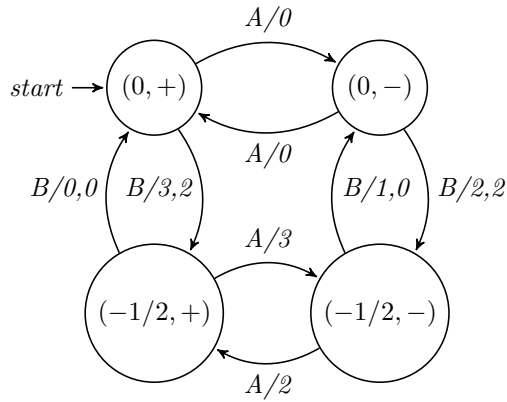
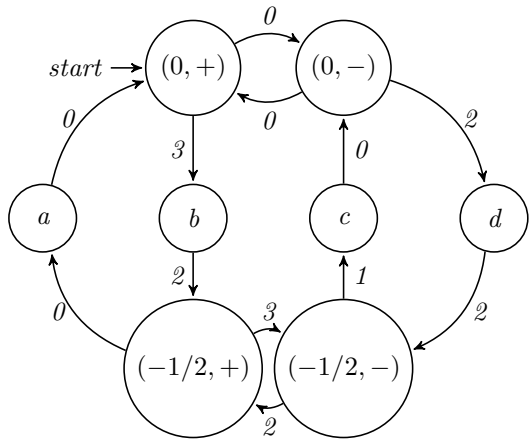

Labeling the vertices of the DFA in the order, $(0,+), (0,-), (1/2,+), (-1/2, -), a, b, c, d$, we can next create a $8 \times 8$ adjacency matrix $M$ for this automaton;
\[ M = \begin{bmatrix} 
0 & 1 & 0 & 0 & 0 & 1 & 0 & 0 \\
1 & 0 & 0 & 0 & 0 & 0 & 0 & 1 \\
0 & 0 & 0 & 1 & 1 & 0 & 0 & 0 \\
0 & 0 & 1 & 0 & 0 & 0 & 1 & 0 \\
1 & 0 & 0 & 0 & 0 & 0 & 0 & 0 \\
0 & 0 & 1 & 0 & 0 & 0 & 0 & 0 \\
0 & 1 & 0 & 0 & 0 & 0 & 0 & 0 \\
0 & 0 & 0 & 1 & 0 & 0 & 0 & 0 
\end{bmatrix}. \]
This has a dominant eigenvalue of $\frac{1+\sqrt{5}}{2}$.  
Hence, utilizing the techniques of \cite{AbramLagarias14a} we have that the Hausdorff dimension of this $5$-adic fractal is $\frac{\log\left(\frac{1+\sqrt{5}}{2}\right)}{\log(5)} \approx 0.298994$.
\end{example}

Let us conclude this section with the following remark. If one is interested only in the dimension of the attractor (and not, for instance, in the language of its points), there is an easier approach than that given by Theorem \ref{thm:main}. One can modify the contractions by linear transformations to get an ``easy case'' with the same dimension. 

\begin{theorem}
\label{thm:conjugate}
\StandardClause
Then there exists a linear map $L$ so that $L \circ F_i \circ L^{-1} (x) =  (-1)^{b_i} p^{k_i} x + d_i'$ with 
\begin{enumerate}
\item $d_i' \in \Z_p \cap \Z$ \label{case:integer}
\item If $b_i = 0$ then $d_i'\geq 0$, and if $b_i =1$ then $d_i' \leq 0$. \label{case:ltgt}
\item There exists an $i$ such that $d_i = 0$. \label{case:zero}
\end{enumerate}
The attractor for $\{L \circ F_i \circ L^{-1}\}$  is $L(K)$ and $\dimp(K) = \dimp(L(K))$.
\end{theorem}

\begin{proof}
Let $L_1(x) = x + a$ and $F(x) = (-1)^b p^k x + d$.
One can check that 
    $L_1 \circ F \circ L_1^{-1}(x) = (-1)^b p^k x + d + (1-(-1)^b p^k) a$.
Let $c_{b,k} = 1 - (-1)^b p^k$.  
We notice that if $b = 1$ then $c_{b,k} > 0$ and otherwise $c_{b,k} < 0$.
There exists an $a$ such
$L_1 \circ F_i \circ L_1^{-1} (x) =  (-1)^{b_i} p^{k_i} x + d_i'$ with $d_i' \in \D' \subseteq  \Q_p$ finite,
    $d_i'\geq0$ for $b_i=0$, $d_i'\leq 0$ for $b_i=1$, and at least one $d_i' =0$.

We note that we do not at the moment have $d_i' \in \Z \cap \Z_p$.
That is, $L_1$ satisfies property \eqref{case:ltgt} and \eqref{case:zero}, but not necessary \eqref{case:integer}.

Let $L_2(x) = c x$ and $F(x) = (-1)^b p^k x + d'$.
One can check that 
    $L_2 \circ F \circ L_2^{-1}(x) = (-1)^b p^k x + c d'$.
By choosing $c$ as the lcm of the denominators of the $d_i'$ from above, we see that
$L_2 \circ L_1 \circ F_i \circ L_1^{-1} L_2^{-1}  (x) =  (-1)^{b_i} p^{k_i} x + d_i''$ with $d_i'' \in \Z_p \cap \Z$.
That is, these $d_i''$ satisfy condition \eqref{case:integer}.
Further we see that $d_i''$ continues to satisfy \eqref{case:ltgt} and \eqref{case:zero}

Setting $L := L_2 \circ L_1$ gives the desired result.
\end{proof}

\begin{example}
\label{ex:simplify}
Consider the $5$-adic self-similar set given by the two maps $A: x \to 5 x + \frac{1}{2}$ and $B: x \to 5 x + \frac{1}{3}$.  
Taking $L_1: x \to x + \frac{1}{12}$ we get
\begin{align*}
L_1 \circ A \circ L_1^{-1} (x) & = 5 x + \frac{1}{6}, \\
 L_1 \circ B \circ L_1^{-1} (x) & = 5 x.
\end{align*}
Taking $L_2: x \to 6 x$ gives
\begin{align*}
A'(x) := L_2 \circ L_1 \circ A \circ L_1^{-1} L_2^{-1} (x) & = 5 x + 1, \\
B'(x) := L_2 \circ L_1 \circ B \circ L_1^{-1} L_2^{-1} (x) & = 5 x.
\end{align*}
Using the result from \cite{AbramLagarias14a}, we see that the Haussdorff dimension of this fractal is $\tfrac{\log2}{\log5}$. The DFA for the self-similar set given by $\{A, B\}$ and by $\{A', B'\}$ are given in Figure \ref{fig:simplify}.

\begin{figure}
\centering
\begin{subfigure}[t]{0.45\textwidth}
\begin{tikzpicture}[->,>=stealth',shorten >=1pt,auto,node distance=2.5cm,
                    semithick]

  \node[initial,state] (A)              {$0$ };
  \node[state]         (B) [right of=A] {$-1/2$};
  \node[state]         (C) [below of=B] {$-5/6$};
  \node[state]         (D) [below of=C] {$-2/3$};
  \node[state]         (E) [below of=A] {$-1/3$};
  \node[state]         (F) [below of=E] {$-1/6$};

  \path (A) edge [bend left]  node {3} (B)
            edge [bend left]  node {2} (E)
        (B) edge [bend left]  node {0} (A)
            edge [bend left]  node {4} (C)
        (C) edge [bend left]  node {3} (D)
            edge [bend left]  node {2} (B)
        (D) edge [bend left]  node {4} (C)
            edge [loop right]  node {3} (D)
        (E) edge [bend left]  node {1} (F)
            edge [bend left]  node {0} (A)
        (F) edge [bend left]  node {2} (E)
            edge [loop right]  node {1} (F);
\end{tikzpicture}
\caption{DFA for Example \ref{ex:simplify} (unsimplified)}
\label{fig:simplifya}
\end{subfigure}
\begin{subfigure}[t]{0.45\textwidth}
\begin{tikzpicture}[->,>=stealth',shorten >=1pt,auto,node distance=2cm,
                    semithick]

  \node[initial,state] (A) {$0$ };

  \path (A) edge [loop below]  node {0} (A)
            edge [loop above]  node {1} (A);
\end{tikzpicture}
\caption{DFA for Example \ref{ex:simplify} (simplified)}
\label{fig:simplifyb}
\end{subfigure}
\caption{DFAs for Example \ref{ex:simplify}}
\label{fig:simplify}
\end{figure}
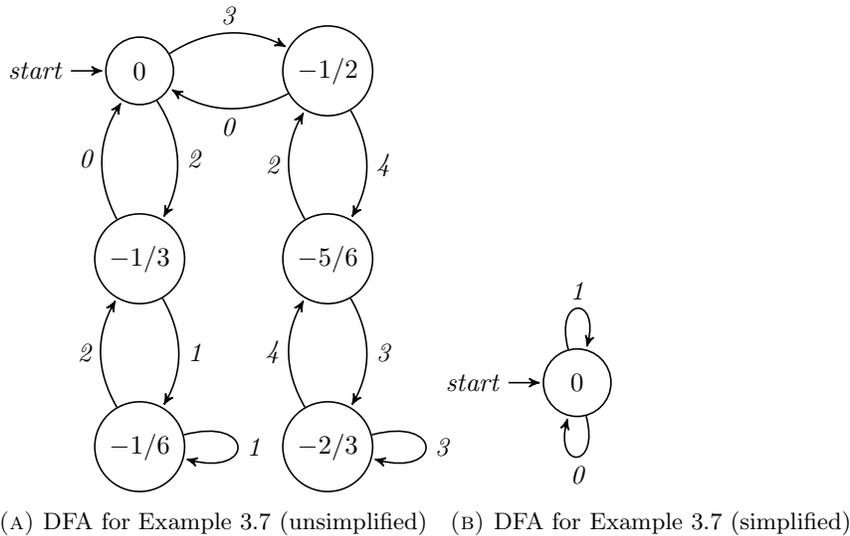

\end{example}

%%%%%%%%%%%%%%%%%%%%%%%%%%%%%%%%%%%%%%%%%%%%
\section{Self-similar sets of dimension 1}
\label{sec:full}

Consider a self-similar set in $\R$. 
It is easy to see that if the set contains an interior point, then it necessarily contains an interval and the set has positive Lebesgue measure and Hausdorff dimension 1.
This is true in higher dimensions as well.  That is, if a self-similar set in $\R^n$ contains an interior point, then it will have positive Lebesgue measure and Hausdorff dimension $n$.

Surprisingly the converse is not true.
In \cite{CJPPS} an example is given of a self-similar set in $\R^2$ which has positive Lebesgue measure, but empty interior.
A more explicit example using countably many maps is given in 
    \cite{BS}.
To the best of the authors' knowledge, it is not known in $\R$ if an example exists of a  self-similar set with positive measure and empty interior.

In this section we will show in the case of $p$-adic self-similar sets whose offsets are in $\Z_p \cap \Q$ that the existence of an interior point is equivalent to the set having Hausdorff dimension $1$.  This in turn is equivalent to the set having positive Haar measure.

The next result is in fact true for all $p$-adic path set fractals, which contain as a subset $p$-adic self-similar sets with offsets in $\Z_p \cap \Q$.

\begin{theorem}
Let $K$ be a $p$-adic path set fractal.
Then the following are equivalent.
\begin{enumerate}
    \item $K$ has Hausdorff dimension $1$ \label{case:1}
    \item The minimal deterministic finite automaton describing the $p$-adic expansions of $K$ has a state $q$ such that there are transitions $q \to_{i} q$ for all $i = 0, 1, \dots, p-1$. \label{case:2}
    \item There exists a finite word $w$ such that the language of $p$-adic expansions of elements of $K$ contains the language $w \{0, 1, \dots, p-1\}^{\N}$. 
    \label{case:3}
\end{enumerate}
\end{theorem}

\begin{proof}
It is clear that \eqref{case:3} implies both \eqref{case:1} and \eqref{case:2}.
Further, \eqref{case:2} implies both \eqref{case:1} and \eqref{case:3}.
Hence it is sufficient to prove that \eqref{case:1} implies either \eqref{case:2} or \eqref{case:3}.

Vertices of an acyclic directed graph can be permuted so that its adjacency matrix is upper triangular. Because a condensation of a graph (a graph of its strongly connected components) is directed acyclic, a permutation of vertices can be found so that the adjacency matrix is upper block triangular. In particular, the eigenvalues are the eigenvalues of the diagonal blocks (in our case corresponding to strongly connected components). Thus, if $K$ has Hausdorff dimension $1$, then at least one of the strongly connected components has to have dimension $1.$ However, a strongly connected component has an irreducible, non-negative adjacency matrix. If such a component has the language $\{0,\dots,p-1\}$, then it has dimension $1$ (and an eigenvalue $p$) and Case \eqref{case:3} is satisfied. If no strongly connected component has the language $\{0, \dots, p-1\}$, then they are all strictly smaller than the block corresponding to the full language.  As such, they will all have dominant eigenvalue strictly less than $p$.
\end{proof}

\section{The Essential class}
\label{sec:essential class}

In Section \ref{sec:self-similar are path set}, we constructed a non-deterministic transducer that outputs the language of expansions of points of a given $p$-adic self-similar set.  From this non-deterministic transducer we then constructed a deterministic finite automaton that accepts this language.

Let $Q$ be the set of states in this deterministic finite automaton.
Following the notation of \cite{F3, HareHareMatthews16} we will say that $Q_L \subseteq Q$ is a {\em loop class} if for all $q_1$ and $q_2$ in $Q_L$ there exists a valid path in $Q_L$ from $q_1$ to $q_2$.  
We will say that $Q_L$ is a maximal loop class if there are no loop classes ${Q_L}'$ with $Q_L \subsetneq {Q_L}'$.
We will say that $Q_L$ is an {\em essential class} if it is a maximal loop class, and further if all paths from $q \in Q_L$ stay in $Q_L$.

In the language of directed graphs, loop classes are also known as strongly connected components, and an essential class is a sink of the condensation of the digraph.

Essential classes have great impact on the study of self-similar sets in $\R$ and $\R^n$, especially self-similar measures (Section \ref{sec:measures}). An 
    important and key results is that, under reasonable conditions, every digraph constructed from a self-similar set has a unique 
    essential class.  Letting $\mu$ be the Hausdorff measure on this $p$-adic self-similar set, we have that almost all points have addresses that are eventually in the essential class. 
We prove the analogous results in this section.

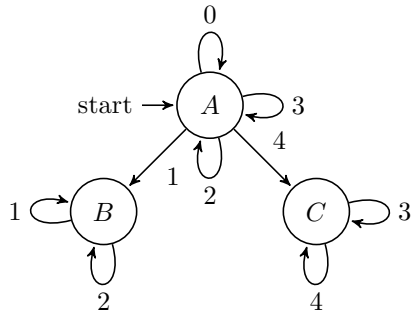
\begin{figure}
\centering
\begin{tikzpicture}[->,>=stealth',shorten >=1pt,auto,node distance=2cm,
                    semithick]

  \node[initial,state] (A)                    {$A$};
  \node[state]         (B) [below left of=A]       {$B$};
  \node[state]         (C) [below right of=A] {$C$};

  \path (A) edge              node {1} (B)
            edge              node {4} (C)
            edge [loop above] node {0} (A)
            edge [loop below] node {2} (A)
            edge [loop right] node {3} (A)
        (B) edge [loop left]  node {1} (B)
            edge [loop below] node {2} (B)
        (C) edge [loop right] node {3} (C)
            edge [loop below] node {4} (C);
\end{tikzpicture}
\caption{Path-set fractal for  Remark \ref{remark:path}}
\label{fig:path}
\end{figure}

\begin{remark}
\label{remark:path}
It is worth noting that this property does not translate for a general path set fractal, as is illustrated in Figure \ref{fig:path}.
This $p$-adic path set fractal has two essential classes, namely $\{B\}$ and $\{C\}$.  Further, the dimension of this path set fractal is $\log(3) / \log(p)$, whereas the dimension of the essential class is $\log(2) / log(p)$.  Hence the almost none (with respect to the Hausdorff measure) are in the Essential class.
\end{remark}

Notice that the adjacency matrix of a finite automaton is a non-negative matrix. Such a matrix has a dominant eigenvalue $\lambda$ corresponding to a non-negative eigenvector, such that $|\lambda|\geq |\lambda'|$ for any other eigenvalue $\lambda'.$ 

We first show that the NDFA constructed in the proof of Theorem \ref{thm:main} has a unique essential class under certain conditions.
\begin{lemma}
\label{lem:UEC}
Let $\F = \{F_i\}$ be a $p$-adic IFS, where each $F_i(x) = (-1)^{b_i} p^{k_i} x + d_i$, with $b_i \in \{0,1\}$, $k_i \in \Z$, $k_i \geq 1$, and $d_i \in \Z_p \cap \Z$.
Let $K$ be the $p$-adic self-similar set associated to $\F$.
Then there exists a NDFA, with a unique essential class, recognizing the language of $p$-adic expansions of $\F$.
\end{lemma}

\begin{remark} Note that here the $d_i \in \Z \cap \Z_p$.
\end{remark}

\begin{proof}[Proof of Lemma \ref{lem:UEC}]
Let $Q$ be the set of states of the directed graph associated to $\F$.
Let $Q^+ \subseteq Q$ be the set of states in the positive orientation and $Q^- \subseteq Q$ be the set of states in the negative orientation.
If $Q^+ = \{(0, +)\}$ we are done.
To see this we note that any combination of maps with an even number of maps with $b_i = 1$ will be in $Q^+$.
Hence, $(0, +)$ will be a descendant of all states, and hence in the essential class.
Further, all descendants of $(0, +)$ will be in the essential class.
As $(0, +)$ is the initial state, this shows that $Q$ is the essential class.

Let $q_{\max} = \max q$ and $q_{\min} = \min q$ where the maximum and minimum are taken
    over all $(q, +) \in Q^+$.
We may assume that one of $q_{\max}$ or $q_{\min}$ is non-zero.
Assume that $q_{\max}  > 0$.  The case where $q_{\min} < 0$ is similar.
There exists a sequence of maps, $F_{a_1}, F_{a_2} \dots, F_{a_k}$ such that
    $F_{a_1} \circ \dots \circ F_{a_n}$ acting on $(0, +)$ will have a final 
    state $(q_{\max}, +)$.
Call this map $F_\omega$.
Let $(q_1, +), (q_2, +) \in Q^+$.
Let $(q_i', +)$ be the resulting state when applying $F_\omega$ to $(q_i, +)$.
We see that if $q_1 < q_2$ then $q_1' \leq q_2'$.
Further we see that if $q_1 < 0$ then $q_1 < q_1'$. 
This implies that repeated applications of $F_\omega$ to any state in $Q^+$ will eventually result in the state $(q_{\max}, +)$.

Let $(q, -) \in Q^-$. 
Let $F_i$ be any map with $b_i = 1$.
We see that the application of $F_i$ to $(q, -)$ will be in $Q^+$.
Hence, by the previous comment, all states in $Q^-$ will have $(q_{\max}, +)$ as a descendant.

As $(q_{\max}, +)$ is a descendant of all states, it is the descendant of all states 
    in the essential class.
Hence it is in the essential class.
Hence all descendants of $(q_{\max}, +)$ are in the essential class.
Hence the essential class is unique.
\end{proof}

\begin{lemma}
\StandardClause
Then there exists a NDFA, with a unique essential class, recognizing the language of $p$-adic expansions of $\F$.
\end{lemma}

\begin{proof}
All directed graphs with positive out-degree for every vertex will have at least one essential class.
As this is the case we are dealing with, it suffices to show that the essential class is unique.

Assume that $q_1$ and $q_2$ are states in possibly different essential classes.
If there exists a map $F_i: x \mapsto (-1)^{b_i} p^{k_i} x + d_i$ with $b_i = 1$,
    the the child of a state under this map will have the opposite orientation of its parent.
As such, we can assume without loss of generality that both $q_1$ and $q_2$ has positive orientation.

We will first show that both $q_1$ and $q_2$ have descendants $q_1'$ and $q_2'$ which are integers
    and in the positive orientation.
Then we will show that the states $q_1'$ and $q_2'$ have a common descendant. The rest then follows from the definition of essential class.

If $q_1$ or $q_2$ are initially integers, we take $q_1 = q_1'$ and $q_2 = q_2'$, as appropriate.

Otherwise, consider a path from the initial state $\mathbf{i}$ to $q_1$.  
This is the image of the transducer acting on some word $\sigma \in \{1,2,\dots,n\}^*$. As $q_1$ is in the positive orientation, we see that the word $\sigma$ acting on a state preserves orientation.
Assume $q_1$ is a state labeled by a non-integer rational with denominator $b_1$. It is worth noting that $p \nmid b_1$.
As $q_1$ is a non-integer in $\Q \cap \Z_p$, we see that $q_1$ has an eventually periodic $p$-adic expansion.
Assume that this expansion has period of length $c_1$ and pre-period of length $d_1$.

Consider the state $q_1'$ given by the image of $\sigma^{b_1 c_1} = \underbrace{\sigma \sigma \dots \sigma}_{b_1 c_1}$ under the transducer.
As this is a descendant of $q_1$, we see that it is in the essential class.
We will next show that $q_1'$, a descendant of $q_1$, is an integer.

We adopt the notation
$[a_0, a_1, a_2, \dots, a_n]$ to mean $\sum_{i=0}^n a_i p^i$, and the analogous notation for 
    infinite $p$-adic expansions.
For eventually periodic $p$-adic expansions, we will use the notation
    $$[a_0, a_1, \dots, a_{d-1}, \overline{a_{d}, \dots, a_{{d}+c-1}}]$$ where
    the $a_0, \dots, a_{d-1}$ is the pre-periodic component of length $d$, and $a_{d}, \dots, a_{d+c-1}$ is the
    periodic component of length $c$. We can also consider different periods (of the same length) and longer pre-periods (containing parts of the period, or even repeated period). Namely, we can write 
 $q_1$ as 
\[ q_1 = [a_0, a_1, \dots, a_{d-1}, a_{d}, \dots ]
          + p^N [\overline{a_{d+j}, \dots, a_{d+c-1}, a_{d},\dots, a_{d+j-1}}]\] for some $N\in\N$,
 $j \in \{0, 1, 2, \dots, c-1\}$, and where the indices of the periodic part are taken modulo $c$ in the range $\{d,\dots,d+c-1\}$.
We note that $[\overline{a_{d+j}, \dots, a_{d+c-1}, a_{d}, \dots, a_{d+j-1}}]$ is 
    a rational number with denominator $b_1$.
Hence we can write $q_1$ as 
\[ q_1 = a_0 + p^N \frac{e_0}{b_1}, \quad a_0, e_0 \in\Z.\]

We can similarly write 
\[ p^{|\sigma|} q_1 = a_1 + p^N \frac{e_1}{b_1}, \quad a_1,e_1\in\Z\]
and, in general 
\[ p^{|\sigma| i} q_1 = a_i + p^N \frac{e_i}{b_1}, \quad a_i,e_i\in\Z\]
for $i \in 0, 1, \dots, b_1 c - 1$.

We see that $e_i = e_{i+c}$ for $i = 0, 1, \dots, b_1 c - 1$.
Hence, for any particular choice of $e^*\in\{e_0,e_1,\dots,e_{b_1c-1}\}$, the number of $i \in \{0,1, \dots, b_1 c - 1\}$ such that $e_i = e^*$ is divisible
    by $b_1$.
Hence the sum of the fractions with $e_i = e^*$ is an integer.
As this is true for all $e^*$ we have that 
\begin{equation}\label{eq:sigma_b_c} q_1 + p^{|\sigma|} q_1 + \dots + p^{|\sigma|(b_1 c_1 - 1)} q_1\end{equation} is an integer, and also, is precisely equal to $q_1'$ defined above as the image of $\sigma^{b_1c}$. To see the latter, remember that $q_1$ is encoded by $\sigma$ (in the input alphabet), thus $q_1+p^{|\sigma|}q_1$ is encoded by $\sigma^2$. In conclusion, \eqref{eq:sigma_b_c} is encoded by $\sigma^{b_1c}.$

We similarly construct $q_2'$, a descendant of $q_2$ and an integer.

This gives us two paths from the initial state $\mathbf{i}$ to two states in the essential class $q_1'$ and $q_2'$ which are 
    both integers.
We next show that $q_1'$ and $q_2'$ have a common descendant, say $q^*$.
If $q_1' = q_2'$ we are done, hence we can assume they are not equal.
We can now use an argument similar to Lemma \ref{lem:UEC} to find a common descendant of $q_1'$ and 
    $q_2'$.
As $q_1$ and $q_2$ are both descendants of all of their descendants (by the definition of 
    an essential class), we see that $q_1$ and $q_2$ are in the same essential class.

This proves that the essential class is unique.
\end{proof}

\begin{example}
\label{ex:5-adic}
Consider the self-similar set given by $p = 5$ and the maps $A: x \to 5 x $ and $B: x \to 5 x - 1/3$.  
The transducer for this self-similar set is given in Figure \ref{fig:5-adic}

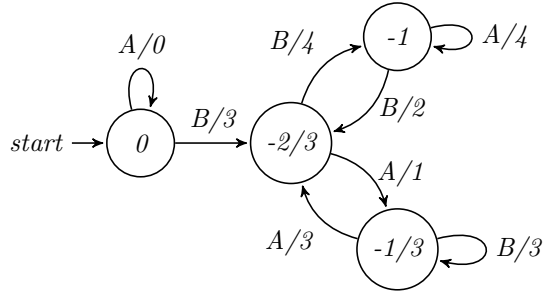
\begin{figure}
\begin{center}
\begin{tikzpicture}[->,>=stealth',shorten >=1pt,auto,node distance=2cm,
                    semithick]

  \node[initial,state] (A)                    {0};
  \node[state]         (B) [right of=A]       {-2/3};
  \node[state]         (C) [above right of=B] {-1};
  \node[state]         (D) [below right of=B] {-1/3};
  
  \path (A) edge              node {B/3} (B)
            edge [loop above] node {A/0} (A)
        (B) edge [bend left] node {B/4} (C)
            edge [bend left] node[right] {A/1} (D)
        (C) edge [loop right] node {A/4} (C)
            edge [bend left] node[right] {B/2} (B)
        (D) edge [loop right] node {B/3} (D)
            edge [bend left] node {A/3} (B);
\end{tikzpicture}
\end{center}
\caption{Transducer for Example \ref{ex:5-adic}}
\label{fig:5-adic}
\end{figure}

As both $A$ and $B$ preserve orientation, we see that all states have positive orientation.
Consider the two states, $-1/3$ and $-2/3$ in the essential class of the transducer.
We see that $-1/3$ is the image of $\sigma_1 = BA$ under the transducer, and $-2/3$ is the image of $\sigma_2 = B$.

As before, we will adopt the notation
$[a_0, a_1, a_2, \dots, a_n]$ to mean $\sum_{i=0}^n a_i p^i$, and the equivalent notation for 
    infinite $p$-adic expansions.
For eventually periodic $p$-adic expansions, we will use the notation
    $$[a_0, a_1, \dots, a_k, \overline{a_{k+1}, \dots, a_{k+n}}]$$ where
    the $a_0, \dots, a_k$ is the per-periodic component, and $a_{k+1}, \dots, a_{k+n}$ is the
    periodic component.

We note that $-1/3 = [\overline{3, 1}] = 3 + 1\cdot 5 + 3 \cdot 5^2 + 1 \cdot 5^3 + 3 \cdot 5^4 + \dots$ and
    $-2/3 = [\overline{1,3}]$.
We see that the period of the $p$-adic expansion of both $-1/3$ and $-2/3$ is of length $2$, the pre-periodis are of length $0$ 
    and the denominator in both cases is $3$.

Let $N = d_1 + |\sigma_1| c_1 b_1 = 0+2 \cdot 2 \cdot 3 = 12$.
We note that 
\begin{align*}
    -1/3 & = [3,1,3,1,3,1,3,1,3,1,3,1] + (-1/3) \cdot 5^{12} 
         & = a_0 + (-1/3)\cdot  5^{12} \\
    -1/3 \cdot 5^2& = [0,0,3,1,3,1,3,1,3,1,3,1] + (-1/3) \cdot 5^{12} 
         & = a_1 + (-1/3)\cdot  5^{12} \\
    -1/3 \cdot 5^4& = [0,0,0,0,3,1,3,1,3,1,3,1] + (-1/3) \cdot 5^{12} 
         & = a_2 + (-1/3)\cdot  5^{12} \\
    -1/3 \cdot 5^6& = [0,0,0,0,0,0,3,1,3,1,3,1] + (-1/3) \cdot 5^{12} 
         & = a_3 + (-1/3)\cdot  5^{12} \\
    -1/3 \cdot 5^8& = [0,0,0,0,0,0,0,0,3,1,3,1] + (-1/3) \cdot 5^{12} 
         & = a_4 + (-1/3)\cdot  5^{12} \\
    -1/3 \cdot 5^{10}& = [0,0,0,0,0,0,0,0,0,0,3,1] + (-1/3)\cdot  5^{12} 
         & = a_5 + (-1/3)\cdot  5^{12} \\
\end{align*}
This gives us that the image of $\sigma_1 \sigma_1 \sigma_1 \sigma_1 \sigma_1 \sigma_1$ is 
\begin{align*}
a_0 + a_1 + \dots 
    +a_5 + (-2) \cdot 5^{12} =\ 
 & [3,1,1,3,4,4,2,1,1,3,4,4] + (-1) \cdot 5^{12}. 
\end{align*}
It is worth noting that the output from the transducer on $(BA)^6 = B, A, B, A, \dots, B, A$ is $3$, $1$, $1$, $3$, $4$, $4$, $2$, $1$, $1$, $3$, $4$, $4$ and 
    we end in state $-1$.

Similarly for $\sigma_2$ we have $b_2 = 3$ and $c_2 = 2$.
Taking $N_2 = d_2 + |\sigma_2| c_2 b_2 = 0 + 1 \cdot 2 \cdot 3 = 6$ we have
\begin{align*}
    -2/3 & = [1,3,1,3,1,3] + (-2/3) \cdot 5^{6} 
         & = a_0' + (-2/3)\cdot  5^{6} \\
    -2/3\cdot 5 & = [0,1,3,1,3,1] + (-1/3) \cdot 5^{6}
         & = a_1' + (-1/3)\cdot  5^{6} \\
    -2/3 \cdot 5^2 & = [0,0,1,3,1,3] + (-2/3)\cdot  5^{6} 
         & = a_2' + (-2/3)\cdot  5^{6} \\
    -2/3 \cdot 5^3 & = [0,0,0,1,3,1]+ (-1/3)\cdot  5^{6} 
         & = a_3' + (-1/3)\cdot  5^{6} \\
    -2/3 \cdot 5^4 & = [0,0,0,0,1,3] + (-2/3)\cdot  5^{6} 
         & = a_4' + (-2/3) \cdot 5^{6} \\
    -2/3 \cdot 5^5 & = [0,0,0,0,0,1] + (-1/3) \cdot 5^{6}
         & = a_5' + (-1/3) \cdot 5^{6} \\
\end{align*}
This gives us that the image of $\sigma_2 \sigma_2 \sigma_2 \sigma_2 \sigma_2 \sigma_2$ is 
    \[ a_0' + a_1' + a_2' + a_3' + a_4' +a_5' + (-3) \cdot 5^6 = [1,4,0,4,0,4] + (-1) \cdot 5^6. \]

It is worth noting that the output from the transducer on $B^6= B, B, B, B, B, B$ is $1, 4, 0, 4, 0, 4$ and we end in state $-1$.

At this point we see that the state $-1$ is a common descendant of both state $-1/3$ and $-2/3$, and hence $-1/3$ and $-2/3$ are in the same
    essential class.
    
If these were different (integer) states then we would have needed to use the additional techniques from Lemma \ref{lem:UEC}.
\end{example}

\begin{definition}
    We say that a point $x \in K$ is {\em essential} or an {\em essential point} if the path through the DFA is eventually in an essential class.
    We say a point $x \in K$ is {\em non-essential} if it is not essential.
\end{definition}

\begin{theorem}\label{thm:dimEC}
\StandardClause
Let $K' \subseteq K$ be the set of non-essential points. 
Let $\mu_{K,p} =H_p^{\dimp(K)}$ be the Hausdorff measure with respect to $K$.
Then the following are true.
\begin{enumerate}
\item $0 < \mu_K(K) < \infty$,  \label{part:finite}
\item $\mu_K(K')=0$,           \label{part:zero}
\item $\dimp(K')<\dimp(K).$    \label{part:smaller}
\end{enumerate}
\end{theorem}

\begin{proof}
As in \cite{AbramLagarias14a}, we define $\iota_p:\Q_p \to \R$ by 
    \[ \iota_p\left(\sum_{i \geq k} a_i p^i\right) = \sum_{i \geq k} \frac{a_i}{p^i}. \]
As $K$ is a path set fractal, we know from \cite{AbramLagarias14a} that
    $\dimp(K) = \dimr(\iota_p(K)) = \log(\lambda)/\log(p)$, where $\lambda$
    is the spectral radius of the adjacency matrix of the path set fractal.
We define $\mu_{0} := H_0^{\dimr(\iota_p(K))}$ as the Hausdorff measure with respect to $\iota_p(K)$.
By \cite{MW} it is known that $0 < \mu_{0}(\iota_p(K)) < \infty$.

We will show that \[ \mu_{0}(\iota_p(K)) \leq \mu_{K,p}(K) \leq 2 p^{\dimr(\iota_p(K))} \mu_{0}(\iota_p(K)) \]
which will prove part \eqref{part:finite} that $\mu_{K,p}(K)$ is positive and finite.

Let $d = \dimp(K) = \dimr(\iota_p(K))$.
Consider
\begin{equation*}
H_{\delta,0}^d(\iota_p(K)) := \inf \left\{ \sum_{i=1}^\infty \diamr(X_i)^d: \iota_p(K) \subseteq \bigcup X_i, \diamr(X_i) < \delta\right\}. 
\end{equation*}
We define 
\begin{align*}
\mathcal{C}_p & = \left\{ a + p^k \Z_p: k \geq 0, 0 \leq a \leq p^{k}-1 \right\} \text{    and} \\
\mathcal{C}_0 & = \left\{ \left[\frac{a}{p^k}, \frac{a+1}{p^k}\right]: k \geq 0, 0 \leq a \leq p^{k}-1 \right\}. 
\end{align*}
Note that for $0 \leq a \leq p^{k}-1$ we have $\iota_p(a + p^k \Z_p) = [\frac{a}{p^k}, \frac{a+1}{p^k}]$,
    hence there is a natural bijection between these two sets.
Hence, a cover of $\iota_p(K)$ restricted to sets in $\mathcal{C}_0$ corresponds to a cover of $K$ in $\Q_p$.

For each $X_i \subset \R$, with $p^{-k} < \diamr(X_i) \leq p^{-k+1}$ there exists two 
    cylinders, $A_i, B_i \in \mathcal{C}_p$ of diameter $p^{-k+1}$ such that
    $X_i \subseteq \iota_p(A_i) \cup \iota_p(B_i)$.
In the case that $X_i \in \mathcal{C}_0$ we can take $A_i = B_i$.
Let $\{X_i\}$ be a cover of $\iota_p(K)$.
Associate to each $X_i$ the pair $A_i, B_i$ such that 
    $X_i \subseteq \iota_p(A_i) \cup \iota_p(B_i)$ with 
    $p^{-k} \leq \diamr(X_i) \leq p^{1-k}$ and $p^{1-k} = \diamp(A_i) = \diamp(B_i)$.
We see that $\{A_i,B_i\}$ is a cover of $K$.
Further, 
\begin{align*}
    \diamr(X_i)^d & \leq \diamp(A_i)^d + \diamp(B_i)^d \\
                  & \leq (p \cdot \diamr(X_i))^d + (p \cdot \diamr(X_i))^d \\
                  & \leq 2 p^d \diamr(X_i)^d.
\end{align*}

This gives us that 
\[ H_{\delta,0}^d(\iota_p(K)) \leq H_{p \delta,p}^d(K) \leq 2 p^d H_{\delta,0}^d(\iota_p(K)).\]
Taking limits gives 
\[ \mu_{0}(\iota_p(K)) \leq \mu_{K,p}(K) \leq 2 p^d \mu_{0}(\iota_p(K)) \]
which proves part \eqref{part:finite} as desired.

To prove part \eqref{part:zero}, that $\mu_{K,p}(K') = 0$ we follow the proof of \cite[Proposition 3.6]{HareHareNg18}.
Notice that $\mu_{K,p}(K') = \lim_{n\rightarrow \infty}\sum \mu_{K,p}(U_n),$ where the sum is taken over the balls $U_n$ of diameter $p^{-n}$ such that $U_n\cap K'\neq\varnothing.$ 
We will show that there exists a $\lambda < 1$ and $N$ such that
\[ \sum \mu_{K,p}(U_{n+N}) \leq \lambda \mu_{K,p}(U_n).\]
Taking limits, this gives that 
\[\mu_{K,p}(K') \leq \lambda \mu_{K,p}(K'),\]
which will prove the result.

Set $N$ equal to the number of vertices in the DFA representation of the $p$-adic 
    self-similar set.
As $\mu_{K,p}$ is a measure, we see that for all $a$ that 
    \[\mu_{K,p}(a + p^n \Z_p) = \sum_{b_n, b_{n+2}, \dots, b_{N+n-1}} 
        \mu_{K,p}(a + b_n p^n + 
             \dots + b_{N+n-1} p^{n+N-1}+p^{N+n} \Z_p).\]
For ease of notation, we write $\bar b = b_n p^n + \dots b_{n+N-1} p^{n+N-1}$.
In particular, for $U_n$ a ball of radius $p^n$ with non-trivial intersection with $K'$ and $a$ such that $a + p^n \Z_p = U_n$ we have
\[ \mu_{K,p}(U_n) = \mu_{K,p}(a+p^n \Z_p) = \sum_{\bar b} \mu_{K,p}(a + \bar b +  p^{N+n}\Z_p). \]
Further, as there are $N$ vertices in the DFA,
    we see that there is at least one choice of $\bar b$ such that
    $(a + \bar b +p^{N+n} \Z_p) \cap K' = \varnothing$ and 
    $(a + \bar b +p^{N+n} \Z_p) \cap K \neq \varnothing$.
This follows as for any state in the DFA, there will exist a path of length at most $N$ which terminates in the 
    essential class.
As $(a + \bar b +p^{N+n} \Z_p) \cap K \neq \varnothing$ and $K$ is $p$-adic self-similar, we see that there is a scaled copy of $K$ 
    inside of $(a + \bar b +p^{N+n} \Z_p) \cap K$,
    and hence $\mu_{K,p}(a + \bar b +p^{N+n} \Z_p) \geq p^{-N} \mu_{K,p}(K)$.
We set $\lambda = (1-p^{-N})$.
Hence, for $U_n$ with $U_n \cap K' \neq \varnothing$, a ball of radius $p^n$, and $U_n = a + p^n \Z_p$ we have that
\begin{align*} 
\mu_{K,p}(U_n) & = \mu_{K,p}(a + p^n \Z_p) \\
               & = \sum_{\bar b} \mu_{K,p}(a + \bar b +  p^{N+n}\Z_p) \\
               & = \sum_{\substack{\bar b \\ a+ \bar b + p^{n+N} \Z_{p} \cap K' \neq \varnothing}} \mu_{K,p}(a+\bar b + p^{N+n} \Z_p) \\
               & \hspace{1cm} + 
               \sum_{\substack{\bar b \\ a+ \bar b + p^{n+N} \Z_{p} \cap K' = \varnothing}} \mu_{K,p}(a+\bar b + p^{N+n} \Z_p) \\
               & = \sum \mu_{K,p}(U_{n+N}) +  
                \sum_{\substack{\bar b \\ a+ \bar b + p^{n+N} \Z_{p} \cap K' = \varnothing}} \mu_{K,p}(a+\bar b + p^{N+n} \Z_p)
\end{align*}

We see in the last line that the second sum is strictly 
    positive and bounded from below by $p^{-N} \mu_{K,p}(a + \Z_p) = p^{-N} \mu_{K,p}(U_n)$.
This gives us that
\begin{align*}
\sum \mu_{K,p}(U_{n+N}) 
    & = \mu_{K,p}(U_n) - \sum_{\substack{\bar b \\ a+\bar b + p^{n+N} \Z_{p} \cap K' = \varnothing}} \mu_{K,p}(a+\bar b + p^{N+n} \Z_p)\\
    & \leq \mu_{K,p}(U_n) - p^{-N} \mu_{K,p}(U_n) \\
    & = \lambda \mu_{K,p}(U_n) 
\end{align*}
as required.

To see part \eqref{part:smaller} we see from the above result that 
   \[ \dimp K' \leq \lambda^{1/N} \dimp(K). \]
As $\lambda < 1$ the result follows.
\end{proof}

Now we can proceed to proving the main statement of this section.
\begin{theorem}\label{thm:EC}
\StandardClause
Then there exists a DFA, with a unique essential class, recognizing the language of $p$-adic expansions of $\F$.
\end{theorem}

\begin{proof}
Let $\A$ be the NDFA associated to the $p$-adic self-similar set with the set of states $Q$.
We let $\Ad$ be the deterministic representation of the $p$-adic self-similar set obtained by the power-set method.
We let $\overline Q$ be the set of states of $\Ad$.
We note that for all $\overline q \in \overline Q$ that $\overline q \subseteq Q$.

We will denote by $EC \subseteq Q$ the unique essential class of the non-deterministic representation of $K$ and $\overline{EC} \subseteq \overline{Q}$ an essential class of the deterministic representation.

We will denote by $\Ln(q)$ the language accepted by an automaton $\A$ with initial state $q$. 
For a set of states $A \subseteq Q$ we define $\Ln(A) = \bigcup_{q \in A} \Ln(q)$.
We define the dimension of a language $\mathcal{L} \subset \{0,1,\dots,p-1\}$ as the dimension of the natural projection of $\mathcal{L}$ to the $p$-adic integers. 

Similarly, we will denote by $\Ld({\overline{q}})$ as the language accepted by the automaton $\Ad$ with initial state $\overline q$. For a set of states ${\overline A} \subseteq {\overline Q}$ we define $\Ld({\overline A}) = \bigcup_{{\overline q} \in {\overline A}} \Ld({\overline q})$.

We see that 
    \[\dimp(\Ln(EC)) = \max_{\overline{EC}} \dimp(\Ld(\overline{EC})) = \dimp(K).\]
The first inequality comes from the fact that the dimension of a finite union
    is the maximum of the dimensions within the union.
The second inequality can be seen from either Part \ref{part:zero} or \eqref{part:smaller} or Theorem \ref{thm:dimEC}.

The second maximum is taken over all possible essential classes in $\Ad$.
It follows from the power-set determinization that for each state $q\in EC$, there is a state $\overline{q} \in \overline{EC}$ such that $q\in \overline{q}$. It then holds that $\Ln(EC)\subseteq \Ld(\overline{EC})$, because $\overline{EC}$ contains every walk of $EC$ (labeling-wise) and possibly more. 
As such, for all choices of $\overline{EC}$ we have that $\dimp(K) = \dimp(\Ln(EC)) \leq \dimp(\Ld(\overline{EC})) \leq \dimp(K)$.
Hence $\dimp(\Ld(\overline{EC})) = \dimp(K)$ for all essential classes $\overline{EC}$.
  
If $\Ln(EC)\neq \Ld(\overline{EC})$, then $\overline{EC}$ contains a path $w_1\dots w_n$, starting in some state $\overline{q}$, that is not a prefix of any member of $\Ln(EC)$. Moreover, as $\overline{EC}$ is a loop class, we see that there exists a $u_1 u_2 \dots u_m$ such that $w_1\dots w_nu_1\dots u_m$ begins and ends in the same state $\overline{q}\in \overline{EC}$. 
There exists a $\overline{q'} \in \overline{EC}$ such that $\overline{q'} \cap EC \neq \varnothing$.
As $\overline{EC}$ is a loop class, we may assume that $\overline{q}$ has this property.
Let $q \in \overline{q} \in \overline{EC}$ such that $q \in EC$.
Construct $\A'$ by appending to the digraph of $\A$ a path from $q$ to $q$ that outputs $w_1 \dots w_n u_1 \dots u_m$.
We see that $EC$ is the essential class of $\A'$.
We define $\Ln'$ on $\A'$ in an analogous way to $\Ln$ on $\A$.
Consider the determinization of $\A'$, which we will denote $\Ad'$.
We define $\Ld'$ on $\Ad'$ in an analogous way to $\Ld$ on $\Ad$.
We have $\Ln(EC)\subsetneq \Ln'(EC') \subseteq \Ld'(\overline{EC}') \subseteq \Ld(\overline{EC})$. 

Let $T, T', \overline{T}', \overline{T}$ be the transition matrices for 
    $EC, EC', \overline{EC}', \overline{EC}$ respectively.
Let $\lambda, \lambda', \overline{\lambda}', \overline{\lambda}$ be the 
    dominant eigenvalue value of these transition matrices.
We see that  $\dimp(K) = \frac{\log(\lambda)}{\log(p)} = \frac{\log(\overline{\lambda})}{\log(p)}$
    and $\lambda < \lambda' \leq \overline{\lambda}' \leq \overline{\lambda}$, a contradiction.
Hence $\Ln(EC) = \Ld(\overline{EC})$ for all $\overline{EC}$.

It remains to be shown that we can identify the essential classes and substitute them for one representative without changing the language of the DFA.

Note that for any DFA, there is a minimal one that is unique (up to an isomorphism). Moreover, this minimal DFA can be obtained by iteratively identifying vertices that are \emph{nondistinguishable} (see Hopcroft's algorithm \cite{Hopcroft}). It is important to notice that by following this procedure, to each ``original'' state $q$, there is a state $\overline{q}$ in the next iteration, such that $\mathcal L(q) = \overline{\mathcal L}(\overline{q}).$ Same is then true for minimal DFA.

Therefore, any essential class $EC$ can be replaced by a minimal representative $EC^*$, and for any $p\notin EC, q\in EC$, an edge $p\rightarrow q$  can be replaced by  $p\rightarrow \overline{q}\in EC^*$ without changing the language this DFA recognizes. It is not hard to see that the resulting DFA is again deterministic, which concludes the proof.

By construction, for each $\overline{q} \in \overline{EC}$ there exists  a $\overline{q}^* \in \overline{EC}^*$
    such that  $\Ld(\overline q) = \Ld^*(\overline{q}^*)$.
Thus, an edge going from the non-essential part of the DFA to $\overline q\in \overline{EC}$ can be switched for an edge going to $\overline{q}^*\in \overline{EC}^*$, without changing the language. Since we can do this with all the states of all the essential classes,  the proof is concluded.

\end{proof}

\section{Decimation}
\label{sec:decimation}

In \cite{AbramLagarias14a}, Abram and Lagarias introduced the concept of decimation of a $p$-adic path set fractal.  In particular, they observed that the class of $p$-adic path set fractals was closed under the operation of decimation.

We define the {\em decimation of a sequence by $k$ with offset $j$} as \[\psi_{j,k}(a_0,a_1,\dots) = (a_j,a_{j+k},a_{j+2k},\dots).\]
  In this section we investigate $p$-adic self-similar sets under the process of decimation.  It is often the case that for sufficiently large $k$, and any offset $j$, that the decimation of a $p$-adic self-similar set will result in a language which is maximal.  That is, let  $E$ be the set of digits that occur in the language of the essential class of a $p$-adic self-similar set $K$.  We often have, for sufficiently large $k$, that the dimension is $\log(\#E)/\log(p)$.

\begin{theorem}
\label{thm:coprime}
\StandardClause
Consider the NDFA associated to the self-similar set $K$.
Assume that there exists a $q$ belonging to an essential class $EC$ such that there are two paths starting and ending at $q$ with co-prime length.
Let $E$ be the set of digits that occur in $\Ln(EC)$.
Then for sufficiently large $k$ we have the Hausdorff dimension of $\phi_{j,k}(K)$ is 
    $\log(\#E)/\log(p)$.
\end{theorem}

\begin{remark}
It is worth remarking here that the length of the path is measured by the length of the 
    output, not by the number of edges traversed.
\end{remark}

\begin{corollary}
\StandardClause
Let $T$ be the transition matrix of the essential class of the deterministic 
    automaton.
Assume there exists a $m$ such that $T^m$ is strictly positive.
Let $E$ be the set of digits that occur in $\Ln(EC)$.
Then for sufficiently large $k$ we have the Hausdorff dimension of $\phi_{j,k}(K)$ is 
    $\log(\#E)/\log(p)$.
\end{corollary}

\begin{proof}
This follows from noting that if $T^m$ is strictly positive, then so is $T^{m+1}$.
As such, for all states in the essential class we have that there exists paths of both lengths $m$ and $m+1$ which both start and end at the same state.  Hence the conditions of Theorem \ref{thm:coprime} are satisfied and the result follows.
\end{proof}

\begin{corollary}\label{coro:decimation}
Let $\F = \{F_i\}$ be a $p$-adic IFS, where each $F_i: x \mapsto p x + d_i$, with $d_i \in \Z_p \cap \Z$.
Let $K$ be the $p$-adic self-similar set associated to $\F$.
Let $E$ be the set of digits that occur in $\Ln(EC)$.
Then for sufficiently large $k$ we have the Hausdorff dimension of $\phi_{j,k}(K)$ is 
    $\log(\#E)/\log(p)$.
\end{corollary}

\begin{remark} 
Note, the shape of the $F_i$ are restricted
\end{remark}

\begin{proof}[Proof of Corollary \ref{coro:decimation}]
As in the proof of Lemma \ref{lem:UEC}, assume that there exists a 
    map $x \mapsto p x + d$ with $d > 0$.
Construct $q_{\max}$ as before.
We see that $(q_{\max}, +)$ is in the essential class, and has a map of length $1$ to itself.
This proves the result.
\end{proof}

\begin{proof}[Proof of Theorem \ref{thm:coprime}]
Let $a \in E$.
Let $q \in EC$ such that there are two paths of co-prime length from $q$ to $q$.
Call these two paths $p_1$ and $p_2$.
There will exist a $q_{2,a}$ such that there is a path from $q$ to $q_{2,a}$ with final output $a$.
Call this path $p_{3,a}$.
In addition, there will exist a path from $q_{2,a}$ to $q$.
Call this path $p_{4,a}$.

For each $a$, considering the set of paths from $q$ to $q$ of the form $\{p_1, p_2\}^* p_{3,a} p_{4,a} \{p_1, p_2\}^*$.
As $p_1$ and $p_2$ are co-prime, we see that for sufficiently large $k$,
    there exists a $j$ (dependent on $n$) such that for all $a$ there
    exists a path of length $k$ and with output $a$ in position $j$.
Further, for sufficiently large $N$ we have that the state is 
    reachable for all $j' \geq N$.
This shows that $\phi_{j+j', k}(K)$ has dimension $\log \#E / \log p$
    as required.
\end{proof}

\begin{example}
\label{ex:not-coprime}
If the conditions of Theorem \ref{thm:coprime} are not met, the result need not follow.
Let $p = 3$.
Consider the two maps $A: x \mapsto 9 x + 3$ and $B: x \mapsto 9 x + 6$.  
The non-deterministic and deterministic automatons for this attractor are shown in Figure \ref{fig:not-coprime}.

\begin{figure}
\centering
\begin{subfigure}[t]{0.4 \textwidth}
        \begin{tikzpicture}[->,>=stealth',shorten >=1pt,auto,node distance=2.8cm,
                    semithick]

  \node[initial,state] (A)              {$0$};
  
  \path (A) edge [loop above] node {$A/01$} (A)
            edge [loop below] node {$B/02$}(A);
\end{tikzpicture}
\caption{Transducer}
\label{fig:not-coprimea}
\end{subfigure}
\begin{subfigure}[t]{0.4 \textwidth}
        \begin{tikzpicture}[->,>=stealth',shorten >=1pt,auto,node distance=2cm,
                    semithick]

  \node[initial,state] (A)              {$0$};
  \node[state]         (A3)  [above of=A]           {a};
  \node[state]         (A6)  [below of =A]            {b};
  
  \path (A) edge [bend left] node {$0$} (A3)
            edge [bend left] node {$0$} (A6)
         (A3) edge [bend left] node {$1$} (A)
         (A6) edge [bend left] node {$2$} (A);
\end{tikzpicture}
\caption{DFA}
\label{fig:not-coprimeb}
\end{subfigure}
\caption{Transducer and DFA for  \ref{ex:not-coprime}}
\label{fig:not-coprime}
\end{figure}
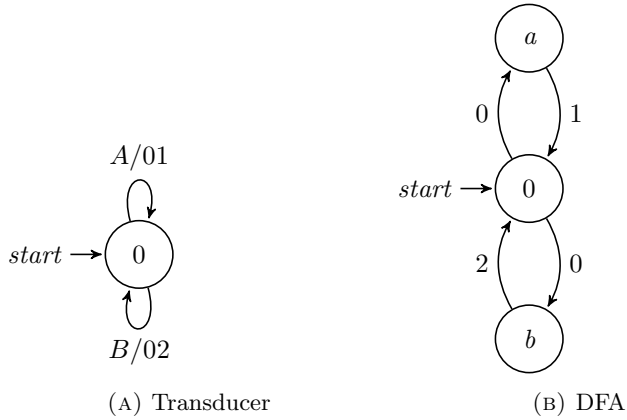

We see that all paths that start and end at the same state are of even length.
Hence we do not have two paths of co-prime length, and the conditions of Theorem \ref{thm:coprime} are not satisfied.
For any even number, the resulting decimation of the set either has dimension $0$ or dimension $\log(2)/\log(3)$, 
    depending on the parity of the offset.  
The dimension of $K$, as well as the decimation by an odd number, independent of the offset, is $\log(2) / \log(9)$.
\end{example}

\begin{example}
\label{ex:decimation}
Consider the self-similar $3$-adic fractal given by the two maps $A: x \mapsto 3x +1$ and $B: x \mapsto 3 x + 5$.
One can quickly compute the transducer (see Figure \ref{fig:quick}).

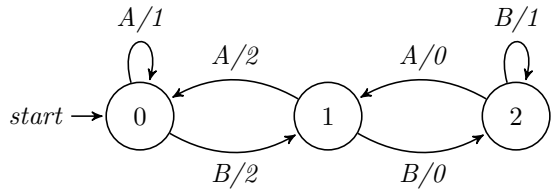
\begin{figure}
\begin{center}
\begin{tikzpicture}[->,>=stealth',shorten >=1pt,auto,node distance=2.5cm,
                    semithick]

  \node[initial,state] (A)                    {$0$};
  \node[state]         (B) [right of=A] {$1$};
  \node[state]         (C) [right of=B] {$2$};

  \path (A) edge [bend right] node[below] {B/2} (B)
            edge [loop above] node {A/1} (A)
        (B) edge [bend right] node[above] {A/2} (A)
            edge [bend right] node[below] {B/0} (C)
        (C) edge [loop above] node {B/1} (C)
            edge [bend right] node[above] {A/0} (B);
\end{tikzpicture}
\end{center}
\caption{Transducer for Example \ref{ex:decimation}}
\label{fig:quick}
\end{figure}

A quick check shows that this is the deterministic FA when removing the input, as there is no non-deterministic output.

The adjacency matrix is \[ T := \begin{bmatrix} 1 & 1& 0 \\ 1 & 0 & 1 \\ 0 & 1 & 1 \end{bmatrix} \] whose dominant eigenvalue is $2$. Hence the dimension of this set is $\dimp(K) = \log(2)/\log(3) \approx 0.63092$.

We see that $T^2>0$, and further that $E = \{0,1, 2\}$ are all the possible labels of the edges.
Hence, by Corollary \ref{coro:decimation}  we see for sufficiently large $k$ that the decimation by $k$, independent of $j$ will
    have $\dimp(\psi_{j,k}(K)) = \log(3)/\log(3) = 1$.
Take for example $k = 3$.
Consider the composition of maps $B\circ B\circ A$.
We see that, regardless of the starting state, that the final output will be $0$ and the final state will be $1$.
Similarly, under the map $B \circ B \circ B$ the final output will be $1$ and 
    under the map $A \circ A \circ B$ the final output will be $2$.
Hence the decimation by $3$, independent of the offset, will have dimension $1$.
\end{example}

\begin{remark}
\label{rmk:C}
It is worth noting that properties of decimation are not invariant under scalar multiplication and translations.  
This is also true in $\R$.
For example, consider the middle third Cantor set $C \subset [0,1]$, written in base $3$.
We see that the decimation of $C$, independent of both $k$ and $j$, will again give $C$, and
    will hence always have dimension $\log(2)/\log(3)$.
Consider instead a scaled shifted version of $C$, namely $C' := \frac{1}{2} C + \frac{1}{4}$.
The directed graph for this path set fractal is given in Figure \ref{fig:C}.

We have that 
\begin{enumerate}
\item If $k \geq 2$, $k$ even and $j$ odd then $\dimr(\psi_{k,j}(C')) = \frac{\log{3}}{\log{3}} = 1$.
\item If $k \geq 2$, $k$ even and $j$ even then $\dimr(\psi_{k,j}(C')) = \frac{\log{2}}{\log{3}}$.
\item If $k \geq 3$, $k$ odd then $\dimr(\psi_{k,j}(C')) = \frac{\log{6}}{\log{9}}$.
\end{enumerate}
\end{remark}

\begin{figure}
\begin{center}
\begin{tikzpicture}[->,>=stealth',shorten >=1pt,auto,node distance=2.5cm,
                    semithick]

  \node[initial,state] (A)                    {};
  \node[state]         (B) [right of=A] {};
  \node[state]         (C) [above of=A] {};
  \node[state]         (D) [below of=A] {};

  \path (A) edge [bend right] node[below] {1} (B)
            edge [bend right] node[right] {0} (C)
            edge [bend right] node[left] {2} (D)
        (B) edge [bend right] node[above] {0} (A)
            edge              node[above] {1} (A)
        (C) edge [bend right] node[left] {1} (A)
        (D) edge [bend right] node[right] {0} (A);

\end{tikzpicture}
\end{center}
\caption{DFA for Remark \ref{rmk:C}}
\label{fig:C}
\end{figure}

\section{\texorpdfstring{$p$}{p}-adic self-similar measures}
\label{sec:measures}

There is a well established literature on self-similar measures $\mu$ on $\R^n$ with support equal to a self-similar set.
See for instance \cite{HareHareMatthews16}.
Let $F_0(x) = \beta x$ and $F_1(x) = \beta x + 1-\beta$.
Let $0 < p < 1$.
A common and classic example is $\mu_\beta$, defined as the 
    unique (up to scaling) measure satisfying 
    \[ \mu_\beta = p \mu_\beta \circ F_0^{-1} + (1-p) \mu_\beta \circ F_1^{-1}.\]
If $p = 1/2$ then this is known as unbiased, otherwise it is known as biased.

If $\beta < 1/2$ then this is a Cantor measure with support on a Cantor set $K$ satisfying 
    $K = F_0(K) \cup F_1(K)$.
If $\beta = 1/2$ and $p = 1/2$ then this is the standard Lebesgue measure restricted to 
    $[0,1]$.
If $1/2 < \beta < 1$ then this is known as a Bernoulli convolution and has been extensively studied \cite{BF, BHM, F3, F1, F4, F2, FLW}.

Of particular interest is the local dimension of a point with respect to a self-similar measure.
\begin{definition}
\label{defn:local dim}
    Let $\mu$ be a measure and $x \in \mathrm{supp}(\mu)$.  
    We define the upper local dimension at $x$ with respect to $\mu$ as
    \[ \overline{\dimlc} \mu(x) = \limsup_{\epsilon \to 0} \frac{\log(\mu(B_\epsilon(x))}{\log(\epsilon)}.\]
Replacing $\limsup$ with $\liminf$ gives the lower local dimension.
If the upper and lower local dimension are equal, then we say that this is the local dimension.
\end{definition}

The goal of this section is to demonstrate that a number of the techniques and results from self-similar measures on $\R$ carry over to $p$-adic self-similar measures in a natural way, with respect to the upper and lower local dimension.

\begin{definition}
    A $p$-adic measure $\mu$ is an additive map from the set of compact open sets $C$ of $\Q_p$ to $\R^+$.
    That is, if $U_1, U_2, \dots, U_n$ is a set of disjoint open compact sets, then 
    \[ \mu\left(\bigcup_{i=1}^n U_i\right) = \sum_{i=1}^n \mu(U_i).\]
\end{definition}

\begin{remark}
It is worth noting that Definition \ref{defn:local dim} is well defined for such measures.
We note that $B_\epsilon(x) = \{y : |x - y|_p < \epsilon\}$ are cylinders of $p$-adic numbers,
\end{remark}

Consider the set of cylinders $a + p^N \Z_p$ in $\Q_p$. We see that these sets are both open and closed, and form a basis of a topology 
    for $\Q_p$.

The most common measure on $\Q_p$ is the Haar measure, defined by $\mu(a + p^N \Z_p) = 1/p^N$.

We can construct a $p$-adic self-similar measures in a similar way to their counterpart on $\R$.
Let $F_0, F_1, \dots, F_n$ be a series of contractions from $\Q_p \to \Q_p$.
Let $p_0, \dots, p_n \in \R$ with $0 < p_i < 1$ and 
     $p_0 + \dots + p_n = 1$.
We define a measure $\mu$ such that 
\begin{equation}
    \mu = p_0 \mu \circ F_0^{-1} + \dots + p_n \mu \circ F_n^{-1}.
    \label{eq:distribution}
\end{equation}
It is further convenient to normalize this so that $\mu(K) = 1$ where $K$ is the attractor of $\{F_1, F_2, \dots, F_n\}$.
We define the (upper, lower) local dimensions for $\mu$ as before.

\begin{example}
The Haar measure restricted to $\Z_p$ is a $p$-adic self-similar measure given by setting $F_i(x) = p x + i$ for $i = 0, 1, \dots, p-1$ and     with equal probabilities $p_i = 1/p$.
In this case all points have local dimension $1$.
\end{example}

\begin{example}
Let $F_0(x) = 3 x +0$ and $F_1(x) = 3 x + 2$ be maps from $\Q_3 \to \Q_3$.
Let $p_0 = p_1 = 1/2$.
Then the measure $\mu = p_0 \mu \circ F_0^{-1} + p_1 \mu \circ F_1^{-1}$ is then natural analog 
    of the Cantor measure in the $3$-adics.
For a cylinder $C := c_0 + c_1 3 + \dots c_{k-1} 3^{k-1} + 3^k \Z_3$ we see that
    $\mu (C) = 1/2^{k}$ if $c_0, c_1, \dots, c_{k-1} \in \{0, 2\}$ and $0$ otherwise.
For all points in the support of $\mu$ we have that the local dimension is $\log(2)/\log(3)$.
\end{example}

For more complicated measures, we often have a range of possible local dimensions, instead of a singleton value.

If $w_1, w_2 \in A^*$ then we define $w_1 w_2$ as the concatenation of $w_1$ with $w_2$, and
    $w_1^k$ as the $k$-fold concatenation of $w_1$ with itself. 
Let $w = a_1 a_2 \dots a_k \in A^*$.
We define $F_w = F_{a_1} \circ \dots \circ F_{a_k}$.
Associated to each $F_i$ is a probability $p_i$ such that $\sum p_i = 1$.
We define $p_w = p_{a_1} \dots p_{a_k}$.

Assume that the $p$-adic IFS (and hence measure) is 
    defined by equicontractive maps $F_i$.
That is, all maps $F_i = p x + d_i$ for some $d_i \in \Z_p \cap \Q$.
A more complicated construction is possible for non-equicontractive by adapting the technique of \cite{HareHareSimms18}.

Let $K$ be the associated $p$-adic self-similar set.
We know that $K = \bigcup F_i (K)$.
It is not difficult to show that for any fixed $k$ that $K = \bigcup_{w \in A^k} F_w(K)$.

We see that if $w \in A^k$ then $F_w(K) \subseteq b + p^k \Z_p$ for some $b = b_1 b_2 \dots \in \Z_p$.
Further, as the centre of every ball of the form $b + p^k \Z_p$ depends 
    only on the first $k$ terms of $b$, we may assume $b = b_1 b_2 \dots b_{k}$. 
This gives us that $F_w(K) \cap b + p^k \Z_p \neq \varnothing$ if and only if $F_w(K) \subseteq b + p^k \Z_p$.
This greatly simplifies our analysis.
By a recursive application of equation \eqref{eq:distribution}, using a proof similar to \cite[Lemma 3.5]{HareHareMatthews16} or \cite{F4}, we have that
\begin{equation} 
\mu \left( (b+p^k \Z_p) \cap K \right) =
    \sum_{w \in A^k, F_w(K) \cap (b+p^k \Z_p) \neq \varnothing} p_w. 
    \label{eq:recursive}
\end{equation}

We will show that there exists a finite set of matrices, such that  the measure of $a + p^k \Z_p$ is the sum of the entries of the product of $k$ of these matrices. The product is explicitly determined by $a$.

With $\mu$, and $p_i$ defined as above,
\[ \mu(a + p^k \Z_p) = \sum_{c=0}^{p-1} \mu (a + c \cdot p^k + p^{k+1} \Z_p). \]

Recall when we constructed the non-deterministic automaton that we labelled the states based on the  remainder.
Then when we constructed our deterministic (albeit not-necessarily minimal) automaton, we labelled the   states as subsets of the set of remainders.
Following \cite[Section 3.2]{HareHareMatthews16} we use these subsets of the set of remainders for the start state and end state of a transition to index the 
   rows and columns of the transitions matrices.

Consider two states in the deterministic automaton, say $q_1 := \{r_1,  \dots, r_k\}$ and 
    and $q_2 := \{r_1',  \dots, r_{k'}'\}$ such that there is a transition from 
    $q_1$ to $q_2$ by output $a \in \{0,1, \dots, p-1\}$.
We will construct a $k \times k'$ matrix $T := T(q_1, q_2)$.
We define 
\[ T[i,j] := \sum_{k \in K(i,j)} p_k \]
where the $K(i,j)$ is the set of all transitions $S_k$ from states $r_i$ to $r_j'$ and with
    output of $a$.
If this set is empty, then the sum is $0$.

It is clear by construction that if $q_1 \to q_2 \to q_3$ then the transitions 
    matrices $T(q_1, q_2)$ and $T(q_2, q_3)$ have compatible dimensions for 
    matrix multiplication.
As both the NDFA and DFA contain a finite number of states, there are a finite number
    of matrices.
Such a measure is said to satisfy the {\em finite type condition}.

We will next show that the measure of a cylinder $a + p^k \Z_p$ can be 
    determined by these matrix multiplication.

In a deterministic automaton, there is at most one edge associated to a particular output.
As such, we see that if for some $w = a_1 a_2 \dots a_k$ we have output 
    $b_0 b_1 \dots b_{k-1}$ we can determine exactly what the final state is in the DFA.
We further see that this final state depends only on 
    $b_0 b_1 \dots b_{k-1}$.
Let $S_{k-1}$ be the state associated to the output $b_0 b_1 \dots b_{k-2}$ and 
    $S_k$ the state associated to the output $b_0 b_1 \dots b_{k-1}$.
We are interested in the transition from $S_{k-1}$ to $S_k$.
Assume that $F_w(K) \subseteq b + p^k \Z_p$.
If $w = a_1 a_2 \dots a_k$ we define $w^- = a_1 a_2 \dots a_{k-1}$.
We see that in the non-deterministic automaton that $F_{w^-}$ is associated to a particular 
    carry state, say $r_i \in S_{k-1}$, and $S_{w}$ is associate
    to a carry state $r_j' \in S_k$.  
The weight contributed by this carry state is $p_{a_k}$ where $w = a_1 a_2 \dots a_k$.
Hence it is $p_{a_k}$ times the weight associated to the carry state $r_i$ in 
    $b_0 b_1 \dots b_{k-2}$.

\begin{example}
\label{ex:measure}
Consider a measure $\mu$ with support the $3$-adic self-similar set of Example \ref{ex:3adic}.  That is:
\begin{itemize}
    \item $A: x \mapsto 3 x$ with probability $p_0$
    \item $B: x \mapsto 3 x + 1$ with probability $p_1$
    \item $C: x \mapsto 3 x + 3$ with probability $p_3$.
\end{itemize}

See Figure \ref{fig:measure} for a visual representation of the transducer with associated probabilities, and the DFA with the associated
    transition matrices.

The local dimension at a point $x = a_0 a_1 a_2 a_3 \dots$ can be computed by using the norm of the matrix product, normalized by the length, as before.
For example, the local dimension at $-1/8 = [\overline{1, 0}]$ would be given by
\begin{align*}
    \dimlc \mu(x)
    &= \lim_{n\to \infty} \frac{\log\left( \left \Vert
    \begin{bmatrix} p_1 \end{bmatrix}
    \begin{bmatrix}p_0 & p_1 \end{bmatrix}
    \left(
    \begin{bmatrix} p_1 & 0 \\ p_0 & p_2 \end{bmatrix}
    \begin{bmatrix} p_0 & p_2 \\ 0 & 0 \end{bmatrix} \right)^{(n-2)/2} \right\Vert \right ) }{\log(1/3^n)} \\
    & = -\frac{\log(p_0 (p_1 + p_2))}{2 \log 3}.
\end{align*}

\begin{figure}
\centering
\begin{subfigure}[t]{0.45 \textwidth}
\begin{tikzpicture}[->,>=stealth',shorten >=1pt,auto,node distance=2.8cm,
                    semithick]

  \node[initial,state] (A)                    {$0$};
  \node[state]         (B) [below of=A] {$1$};

  \path (A) edge [bend right=50]  node[left] {$C/0/p_2$} (B)
            edge [loop above] node {$A/0/p_0$} (A)
            edge [loop right] node {$B/1/p_1$} (A)

        (B) edge [bend left=20] node[right] {$A/1/p_0$} (A)
            edge [bend right=70] node[right] {$B/2/p_1$} (A)
            edge [loop right] node {$C/1/p_2$} (B);

\end{tikzpicture}
\caption{Transducer}
\label{fig:measurea}
\end{subfigure}
\begin{subfigure}[t]{0.45 \textwidth}
\begin{tikzpicture}[->,>=stealth',shorten >=1pt,auto,node distance=3.5cm,
                    semithick]

  \node[initial,state] (A)                    {$\{0\}$};
  \node[state]         (B) [below of= A] {$\{0,1\}$};

  \path (A) edge [bend left] node {$\begin{bmatrix}p_0 & p_1 \end{bmatrix}$,  0} (B)
            edge [loop right] node {$\begin{bmatrix} p_1 \end{bmatrix}$,  1} (A)
        (B) edge [loop right] node {$\begin{bmatrix} p_0 & p_2 \\ 0 & 0 \end{bmatrix}$,  0} (B)
            edge [loop left] node {$\begin{bmatrix} p_1 & 0 \\ p_0 & p_2 \end{bmatrix}$,  1} (B)
            edge [bend left] node {$\begin{bmatrix} 0 \\ p_1 \end{bmatrix}$,  2} (A);
\end{tikzpicture}
\caption{DFA}
\end{subfigure}
\caption{Transducer and DFA with associated probabilities and transition matrices for Example \ref{ex:measure}}
\label{fig:measure}
\end{figure}
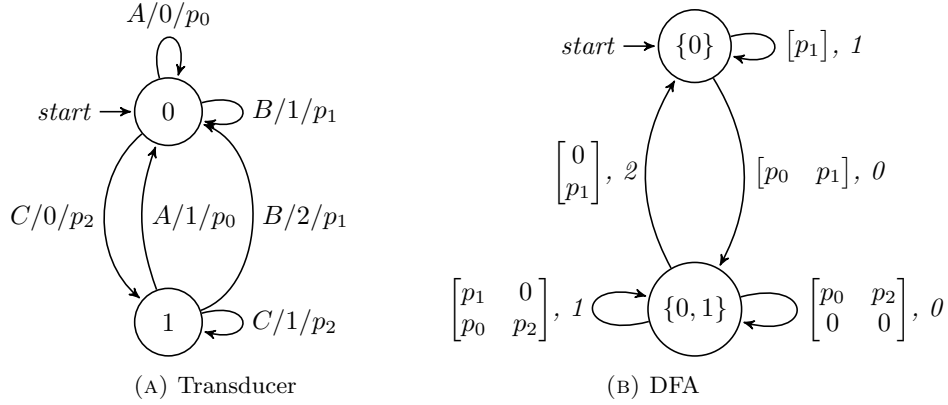

\end{example}

As the above example shows, the methods and technique from \cite{F3, F1, F4, F2, HareHareMatthews16} can be extended to $p$-adic self-similar measures in a natural way.
The upper and lower local dimension of points in $\Z_p$ can be computed using 
    similar techniques.
In $\R$ the computation of the upper local dimensions is complicated by 
    the fact $\limsup \log(\mu [x-\epsilon, x])/\log(\epsilon)$ need not equal
    $\limsup \log(\mu[x,x+\epsilon])/\log(\epsilon)$.  
A similar comment holds for (lower) local dimensions.
This is not an issue in the $p$-adic case, as all points are the centre of the cylinder
    in which they are contained.

Adapting \cite{HareHareNg18} and Theorem \ref{thm:dimEC} we get
\begin{prop}
\StandardClause
Let $\mu$ be a self-similar measure defined on $K$ as in \ref{eq:distribution}
Let $K' \subseteq K$ be the set of points outside of the essential classes. 
Then the following are true.
\begin{enumerate}
\item $0 < \mu < \infty$  
\item $\mu(K')=0$.        
\end{enumerate}
\end{prop}

We say that $\mu$ has the positive row property if every row of every transition matrix has a non-zero entry.  
We say that a point is periodic if it is $p$-adic representation is eventually periodic.
We say that a point $x$ is a positive periodic point if it is periodic and the 
    transition matrix associated to the period is has strictly positive entries.

Adapting the proofs of \cite{HareHareNg18} we get the following results for the $p$-adic self-similar measures.

\begin{theorem}[Analogous to Theorem 3.12 of \cite{HareHareNg18}]
Suppose $\mu$ is a $p$-adic self-similar measure satisfying the positive row property. Then the set of lower local dimensions of $\mu$ at essential, positive, periodic points is dense in the set of all local dimensions of $\mu$ at essential points. A similar statement holds for the (upper) local dimensions.
\end{theorem}

\begin{theorem}[Analogous to Theorem 3.13 of \cite{HareHareNg18}]
Suppose $\mu$ is a $p$-adic self-similar measure satisfying the positive row property. Assume $(x_n)$ are essential, positive, periodic points. There there is an essential point $x$ such that
\[ \dimlc \mu(x) = \limsup \dimlc \mu(x_n) \]
and
\[ \dimlc \mu(x) = \liminf \dimlc \mu(x_n). \]
\end{theorem}

\begin{theorem}[Analogous to Theorem 3.14 of \cite{HareHareNg18}]
Suppose $\mu$ is a $p$-adic self-similar measure satisfying the positive row property. Let $y,z$ be essential, positive, periodic points. Then the set of local dimensions of $\mu$ at truly essential points contains the closed interval with endpoints $\dimlc \mu(y)$ and $\dimlc \mu(z)$.
\end{theorem}

\begin{corollary}[Analogous to Corollary 3.15 of \cite{HareHareNg18}]
Let $\mu$ be a self-similar measure satisfying the positive row property. Let \[I = \inf\{\dimlc \mu(x)  : x \text{ interior essential, positive, periodic}\}\] and 
\[S = \sup\{\dimlc \mu(x) : x \text{ interior essential, positive, periodic}\}.\] Then
$\{\dimlc \mu(x): x \text{ essential}\} = [I, S]$. A similar statement holds for the lower and upper local dimensions.
\end{corollary}

\begin{theorem}[Analogous to Theorem 3.18 of \cite{HareHareNg18}]
Let $\mu$ be a $p$-adic self-similar measure satisfying the positive row property. Then there exists a essential element $x$ with $\dimlc\mu(x) = \dimp (K)$.
\end{theorem}

It is likely that many other results from self-similar measures will also carry over in a similar or obvious way to $p$-adic self-similar measures.

\section{Open questions \& Comments}
\label{sec:conclusions}

In this paper we demonstrated that $p$-adic self-similar fractals are in fact $p$-adic path set fractal.  
These self-similar sets are all recognizable by a finite state automaton.
We showed that the associated DFA has a unique essential class. 

We gave examples where the contraction factor was $-p^{k_i}$ for some $k_i \geq 1$.  
It should be possible to extend these types of results to algebraic extensions of the $p$-adics.

We next studied decimation.
We gave conditions for when the decimation of a self-similar set has ``maximal'' dimension.
We also gave examples of more general self-similar sets when this was not in fact true.
It would be interesting to explore this more fully.

It was shown in \cite{AbramLagarias14a} that the set of $p$-adic path set fractals is closed under the process of decimation.  As $p$-adic self-similar sets are $p$-adic path set fractals, it is clear that the decimation of a $p$-adic self-similar set is a $p$-adic path set fractal.  It would be interesting to know under what conditions the decimation of a $p$-adic self-similar set is a $p$-adic self-similar set.

We used the self-similar fractals as a basis for creating $p$-adic self-similar measures.  This follows a long history of self-similar measures in $\R$.  We explored local dimension, and showed that it is in fact 
easier to compute in this setting, with fewer complications.
There are a number of questions about self-similar measures that have not been explored, but could lead to interesting results.  The most obvious of which is an exploration of $L^q$-spectrum for $p$-adic self-similar measures.
See for example \cite{LN}.

%\bibliographystyle{plain}
%\bibliography{reference}

\end{document}